\documentclass[sn-mathphys-num]{sn-jnl}% Math and Physical Sciences Numbered Reference Style 
\usepackage{multirow}
\usepackage{booktabs}
\usepackage{graphicx}%
\usepackage{multirow}%
\usepackage{float}
\usepackage{amsmath,amssymb,amsfonts}%
\usepackage{amsthm}%
\usepackage{mathrsfs}%
\usepackage[title]{appendix}%
\usepackage{xcolor}%
\usepackage{textcomp}%
\usepackage{hyperref}%
\usepackage{xr-hyper}%
\usepackage{manyfoot}%
\usepackage{booktabs}%
\usepackage{algorithm}%
\usepackage{algorithmicx}%
\usepackage{algpseudocode}%
\usepackage{listings}%
\usepackage{multirow}
\usepackage{booktabs}
\usepackage{tabularx}
\usepackage{caption} % 加载 caption 宏包
\usepackage{longtable}
\DeclareMathOperator*{\argmin}{arg\,min}
\theoremstyle{thmstyleone}%
\newtheorem{theorem}{Theorem}%  meant for continuous numbers

\theoremstyle{thmstyletwo}%

\theoremstyle{thmstylethree}%

\begin{document}

\title[Article Title]{Incorporating Continuous Dependence Qualifies Physics-Informed Neural Networks for Operator Learning}


\author[1]{\fnm{Guojie}\sur{Li}}

%\author[2,3]{\fnm{Sheng}\sur{Ran}}\email{ransheng@ruc.edu.cn}

\author*[2]{\fnm{Wuyue}\sur{Yang}}\email{yangwuyue@bimsa.cn}

\author*[1]{\fnm{Liu}\sur{Hong}}\email{hongliu@sysu.edu.cn}
\affil[1]{\orgdiv{School of Mathematics}, \orgname{Sun Yat-sen University}, \orgaddress{\city{Guangzhou}, \postcode{510275}, \country{China}}}

%\affil[2]{\orgdiv{Institute of Statistics and Big Data}, \orgname{Renmin University of China}, \orgaddress{\street{No. 59, Zhongguancun Street, Haidian District}, \city{Beijing}, \postcode{100872}, \state{Beijing}, \country{China}}}

\affil[2]{\orgdiv{Beijing Institute of Mathematical Sciences and Applications}, \orgaddress{\city{Beijing}, \postcode{101408}, \country{China}}}


\abstract{Physics-informed neural networks (PINNs) have been proven as a promising way for solving various partial differential equations, especially high-dimensional ones and those with irregular boundaries. However, their capabilities in real applications are highly restricted by their poor generalization performance. Inspired by the rigorous mathematical statements on the well-posedness of PDEs, we develop a novel extension of PINNs by incorporating the additional information on the continuous dependence of PDE solutions with respect to parameters and initial/boundary values (abbreviated as cd-PINN). Extensive numerical experiments demonstrate that, with limited labeled data, cd-PINN achieves 1-3 orders of magnitude lower in test MSE than DeepONet and FNO. Therefore, incorporating the continuous dependence of PDE solutions provides a simple way for qualifying PINNs for operator learning. }




\maketitle

\section{Introduction}\label{sec_introduction}
In many fields of science and engineering, such as computational fluid dynamics, climate prediction, medical imaging, and game simulation, it is often necessary to repeatedly solve complex partial differential equations (PDEs) with varied initial/boundary values and parameter configurations \cite{richard2008methods, raissi2020hidden, tartakovsky2020physics, goswami2020transfer, hennigh2021nvidia}. Traditional numerical solvers, including finite difference (FD), finite element (FE), and finite volume (FV) methods \cite{leveque2007finite, ZIENKIEWICZ20131, leveque2002finite, hughes1987finite}, rely on discretization and thus face an inherent trade-off between accuracy and computational cost: fine grids ensure accuracy but are time-consuming, while coarse grids improve efficiency at the expense of precision. This limitation makes it challenging to achieve both real-time performance and high fidelity in applications, highlighting the need for a more efficient PDE solver.

Deep-learning-based differential equation solvers are generally considered to have a significant potential to improve computational efficiency \cite{esmaeilzadeh2020meshfreeflownet, kochkov2021machine}. One notable method is the physics-informed neural network \cite{raissi2019physics}, which integrates knowledge of governing equations into the loss function to train the model and approximates the solution to PDEs without discretizing the solution domain. Alternative approaches include Deep Galerkin method \cite{sirignano2018dgm}, Deep Ritz method \cite{yu2018deep} and many others \cite{jagtap2020conservative, jagtap2020extended, yu2022gradient}. However, these approaches treat each variation in parameters or initial/boundary values as an independent task, requiring costly retraining. To overcome this limitation, operator-learning methods emerged, including DeepONet \cite{lu2021learning}, Neural Operator\cite{kovachki2023neural}, and Fourier Neural Operators \cite{li2020fourier}, which learn mappings between functional spaces (see Fig.\ref{fig_diagram} A) and enable rapid inference for unseen configurations. While their physics-enhanced variants, such as PI-DeepONet \cite{wang2021learning} and PINO \cite{li2024physics}, partially mitigate the heavy demand for labeled data by incorporating residual losses, data efficiency remains a challenge. More recently, meta-learning-based PINNs have been proposed to improve generalization across parameterized PDEs \cite{finn2017model, antoniou2018train, nichol2018reptile}. They treat variations in parameters and initial/boundary conditions as subtasks within a unified framework, enabling the model to leverage knowledge across related tasks. Feedforward meta-PINNs, such as Hyper-PINN and Meta-MgNet \cite{de2021hyperpinn, chen2022meta}, directly map equation configurations to PINN weights, whereas MAML-based methods, including MAD-PINN (Fig.\ref{fig_diagram} B) and reptile-based PINNs \cite{huang2022meta, liu2022novel}, focus on learning initialization states that adapt efficiently to new tasks. Despite these advances, meta-learning PINNs often suffer from long training time and computationally intensive fine-tuning when applied to a wide range of new configurations.

\begin{figure}[H]
    \centering
    \includegraphics[width=1.0\linewidth]{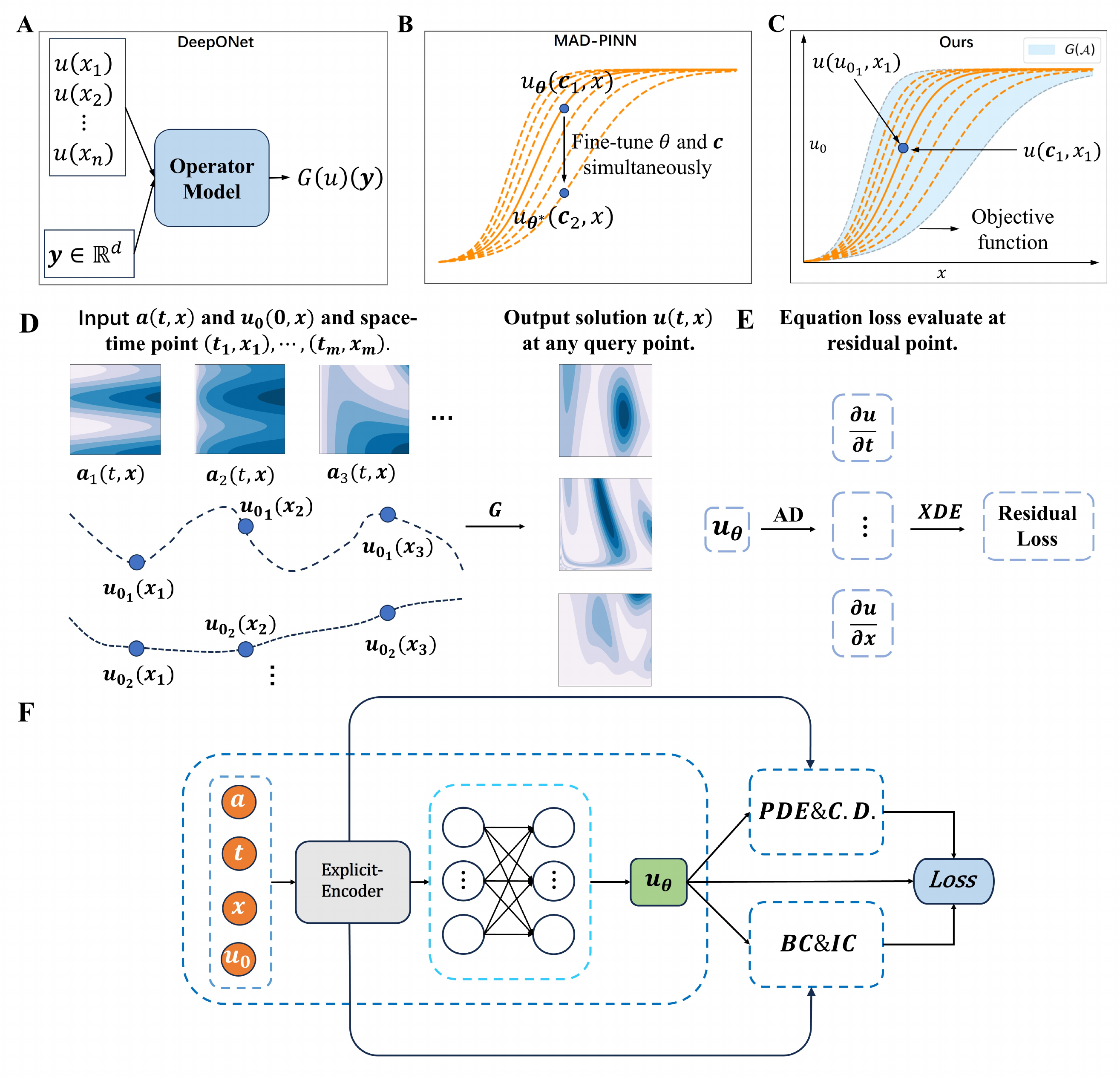}
    \caption{\textbf{Illustration of the idea, problem setup, and architecture of cd-PINN.} (\textbf{A}) Schematic diagram of the input and output of neural operators. (\textbf{B}) MAD learns the solution of the equation under new configurations by fine-tuning the encoding $c$. (\textbf{C}) The objective function of cd-PINN is based on the continuity assumption. (\textbf{D}) Illustration of the labeled training data. For each output $u(t_i, \boldsymbol{x}_i)$, we require the same number of evaluations of $a(t_i, \boldsymbol{x}_i)$ and $u_0(\boldsymbol{x}_i)$ at the same scattered space-time point $(t_i, \boldsymbol{x}_i)$. (\textbf{E}) The flowchart for calculating residual loss. (\textbf{F}) The architecture of cd-PINN.}
    \label{fig_diagram}
\end{figure}

The continuous dependence of PDE solutions on the initial/boundary values and parameters is one of the fundamental requirements for the well-posedness of PDEs\cite{evans2022partial}, yet it is often overlooked in solving parameterized PDEs. To address the challenges of data inefficiency and repeated training in deep-learning-based solvers, we propose the continuous dependence physics-informed neural network (cd-PINN). By using an appropriate encoding $c$ for each configuration (e.g. $u_0$), cd-PINN captures a higher-dimensional solution space $G(\mathcal{A})$ (shown in Fig.\ref{fig_diagram} C) that continuously depends on both the encoding $c$ and the spatial coordinate $x$. This design yields substantially improved data efficiency compared with models like DeepONet and FNO, while avoiding the retraining or fine-tuning typically required by PINN variants at the same time. The numerical results demonstrate the outstanding accuracy, reliability, and practicality of our proposed model for solving large-scale parameterized partial differential equations in real-time. 

\section{Results and Discussion}\label{sec_results}
To validate the outstanding performance of cd-PINN, we conducted numerical experiments on five representative PDEs whose solutions exhibit continuous dependence on their parameters or initial values. These include the three fundamental types of second-order PDEs, including the diffusion, wave, and Poisson equations, as well as high-dimensional diffusion-reaction equations, Burgers, and Navier-Stokes equations. In these examples, we demonstrate the superiority of our method over mainstream operator learning frameworks on diffusion and wave equations, analyze the effect of the loss function $\mathcal{L}_{cd}$ on residual-point density in Poisson equation and on PDE dimensionality in high-dimensional diffusion-reaction equations, and benchmark the efficiency and accuracy of cd-PINN against the traditional finite difference method on the Burgers equation. We also examine the diffusion equation in a setting where the theoretical guarantee of continuous dependence on the initial condition no longer holds. Finally, the proposed model was applied to the tau protein aggregation dynamics in Alzheimer's disease based on both real and synthetic data. In what follows, the negative logarithm of the mean absolute error (NLMAE) is adopted for performance assessment, which is defined as:
\begin{equation}
    \text{NLMAE} = -\log_{10}\Big(\frac{1}{n}\sum_{i=1}^{n}\big|u_{true}^{i}-u_{pred}^{i}\big|\Big).
\end{equation}

\subsection{Benchmarking Against State-of-the-Art Operator Learning Frameworks}\label{sec_dif_wave}
To ensure a fair comparison, we evaluate all models on the same example by fixing the form of the initial value while varying its coefficients to test the predictive performance under continuously changing inputs. First, we examine the 1D diffusion equation, which is described by
\begin{equation}
    \frac{\partial p}{\partial t} - D \frac{\partial^2 p}{\partial x^2} = 0, \; x\in[-10.0, 10.0], t\in [0.1, 1.1]
    \label{eq_diffusion}
\end{equation}
with initial condition $p(x, 0) = \frac{10}{\sqrt{2\pi}\sigma}\exp\Big(-\frac{(x-\mu)^2}{2\sigma^2}\Big)+\frac{10}{\sqrt{2\pi}\sigma_1}\exp\Big(-\frac{(x-\mu_1)^2}{2\sigma_1^2}\Big).$ In practice, we fix $\sigma_1=1.0, \mu_1=5.0$, use $20$ labeled data with $\sigma=1.0, \mu=-5.0$, and $2^{14}$ residual data points to train the model. The test data consists of $101,000$ sets of solutions corresponding to values of $\sigma \in [0.1, 10.0]$ and $\mu \in [-5.0, 5.0]$, totaling $8,080,000$ data points.

We conduct a comprehensive validation of our proposed models against existing baseline models, including FNO and DeepONet. As shown in Fig. \ref{fig_Case_1}A, we compared the test MSE of each model and tracked its evolution across training epochs with only 20 labeled training data points. The results show that the predictions of DeepONet and FNO are significantly worse than those of cd-PINN, PINN, and cd-PINN$^{\#}$, where PINN means no additional differentiability loss $\mathcal{L}_{cd}$ is added, while cd-PINN$^{\#}$ means no explicit encoding is used, but rather encoding is performed using a deep-learning-based encoder (see Supplementary Information). In practice, we experimented with both a MLP-based encoder and an encoder-only Transformer \cite{radford2019language} to encode $a(t, \boldsymbol{x})$ and $u_0(\boldsymbol{x})$, and found the former shows a better performance. In Supplementary Information, we give the results of DeepONet and FNO when labeled training data is 20, 100, 1000, and 10000, respectively, to ensure the correctness of our reproduction model. Panel B in Fig.~\ref{fig_Case_1} shows the NLMAE of DeepONet, PINN, cd-PINN and cd-PINN$^{\#}$ in the phase space defined by parameters $\mu$ and $\sigma$. Each point represents the negative logarithm of the mean absolute error between the predicted solution and the true solution for the corresponding $(\mu, \sigma)$. DeepONet, constrained by sparse labeled data and the lack of continuity exploitation with respect to initial values, exhibits significantly lower predictive accuracy compared with cd-PINN$^{\#}$. Compared with cd-PINN$^{\#}$, if we know the specific form of the initial condition $p_0(x)$, the NLMAE of PINN can be improved by an order of magnitude, thanks to the inclusion of physical information. By further incorporating the differentiability constraint $\mathcal{L}_{cd}$ on the solution coefficients into PINN, the resulting model not only achieves a higher NLMAE overall but also significantly improves its performance in regions where PINN performs poorly -- particularly for small $\sigma$. In Fig.~\ref{fig_Case_1}C, we present the predicted solutions for a new configuration $(\sigma=0.2, \mu=0.0)$ unseen during labeled training data. While DeepONet shows a marked deviation from the ground truth, the cd-PINN and cd-PINN$^{\#}$ models demonstrate superior generalization by accurately recovering the underlying solution morphology.

To make a further verification, we consider the 2D wave system as a special case of the general hyperbolic equations, which is described by
\begin{equation}
    \frac{\partial^2 u}{\partial t^2} = c^2\Big(\frac{\partial^2 u}{\partial x^2} + \frac{\partial^2 u}{\partial y^2}\Big), \; x\in[0, 1], \; y\in [0, 1], \; t\in [0, 0.5],
\end{equation}
with initial condition and boundary conditions
\begin{equation}
    \begin{split}
        & u(x, y, t=0)=10\sin(k x) \sin(k y),\; u_t(x, y, t=0)=0,\; x\in [0, 1], \; y\in[0,0.5],\\
        & u(x=0, y, t) = 0, \; u(x=1, y, t) = 10\sin(k)\sin(ky)\cos(\sqrt{2}ckt), \; y\in[0, 1], \; t\in [0, 0.5],\\
        & u(x, y=0, t) = 0, \; u(x, y=1, t) = 10\sin(kx)\sin(k)\cos(\sqrt{2}ckt), \; x \in [0, 1], \; t\in [0, 0.5].
    \end{split}
\end{equation}

In practice, we use $20$ labeled data with $c=0.505, k=0.505$ (see Fig.~\ref{fig_Case_1}G) and $2^{13}$ residual data points to train the model. The test data consists of $441$ sets of solutions corresponding to values of $c\in [0.01, 1.0]$ and $k\in [0.01, 1.0]$, totaling over $4$ million data points. To ensure the accuracy of the test, both the training data and the test data are generated with a true solution of the equation. The specific data settings are summarized in Table \ref{tab_diffusion_wave}.

\begin{table}[htbp]
	\centering
	\caption{Summary on the setup and results of   diffusion, wave, Poisson, and diffusion-reaction equations}
	\label{tab_diffusion_wave}
	\begin{tabular}{ccccccc}
		\toprule
		Equation & Model & Params & Configurations & Labeled & Residual& NRMSE\\
			\midrule
			\multirow{5}*{Diffusion} & cd-PINN & $\approx 8\times 10^{4}$ & 1 & 20 & 16384 & $5.24\times 10^{-3}$\\
			\cmidrule{2-7}
			& PINN & $\approx 8\times 10^{4}$ & 1 & 20 & 16384 & $2.07\times 10^{-2}$\\
			\cmidrule{2-7}
			& cd-PINN$^{\#}$ & $\approx 9\times 10^{4}$ & 1 & 20 & 16384 & $8.40\times 10^{-2}$\\
			\cmidrule{2-7}
			&DeepONet & $\approx 4\times 10^{4}$ & 1 & 20 & $-$ & $7.80\times 10^{-1}$\\
			\cmidrule{2-7}
			& FNO & $\approx 2\times 10^{5}$ & 1 & 20 & $-$ & $1.98\times 10^{-1}$\\
			\midrule
			\multirow{5}{*}{Wave} & cd-PINN & $\approx 8\times 10^{4}$ & 1 & 20 & 8192 & $1.20\times 10^{-3}$\\
			\cmidrule{2-7}
			& cd-PINN$^{\#}$ & $\approx 1\times 10^{5}$ & 1 & 20 & 8192 & $1.22\times 10^{-2}$\\
			\cmidrule{2-7}
			& PI-DeepONet & $\approx 1\times 10^{5}$ & 1 & 20 & 8192 & $3.43\times 10^{-2}$\\
			\cmidrule{2-7}
			& DeepONet & $\approx 1\times 10^{5}$ & 100 & 20000 & $-$ & $1.06\times 10^{-1}$\\
			\cmidrule{2-7}
			& FNO & $\approx 9\times 10^{4}$ & 100 & 20000 & $-$ & $6.58\times 10^{-1}$\\
            \midrule
		\multirow{2}{*}{Poisson} &     cd-PINN & $\approx 8\times     10^{4}$ & 1 & 20 & 2048 &     $2.68\times 10^{-3}$\\
            \cmidrule{2-7}
            & PINN & $\approx 8\times 10^{4}$ & 1 & 20 & 2048 & $2.14\times 10^{-2}$\\
            \midrule
            \multirow{2}{*}{D-R 2d} & cd-PINN & $\approx 8\times 10^{4}$ & 1 & 100 & 8192 & $2.12\times 10^{-3}$\\
            \cmidrule{2-7}
            & PINN & $\approx 8\times 10^{4}$ & 1 & 100 & 8192 & $1.79\times 10^{-2}$\\
            \midrule
            \multirow{2}{*}{D-R 5d} & cd-PINN & $\approx 8\times 10^{4}$ & 1 & 100 & 8192 & $5.36\times 10^{-3}$\\
            \cmidrule{2-7}
            & PINN & $\approx 8\times 10^{4}$ & 1 & 100 & 8192 & $1.52\times 10^{-2}$\\
            \midrule
            \multirow{2}{*}{D-R 8d} & cd-PINN & $\approx 8\times 10^{4}$ & 1 & 100 & 8192 & $1.61\times 10^{-1}$\\
            \cmidrule{2-7}
            & PINN & $\approx 8\times 10^{4}$ & 1 & 100 & 8192 & $1.85\times 10^{-1}$\\
		\bottomrule
	\end{tabular}
\end{table}

\begin{figure}[H]
    \centering
    \includegraphics[width=1.0\linewidth]{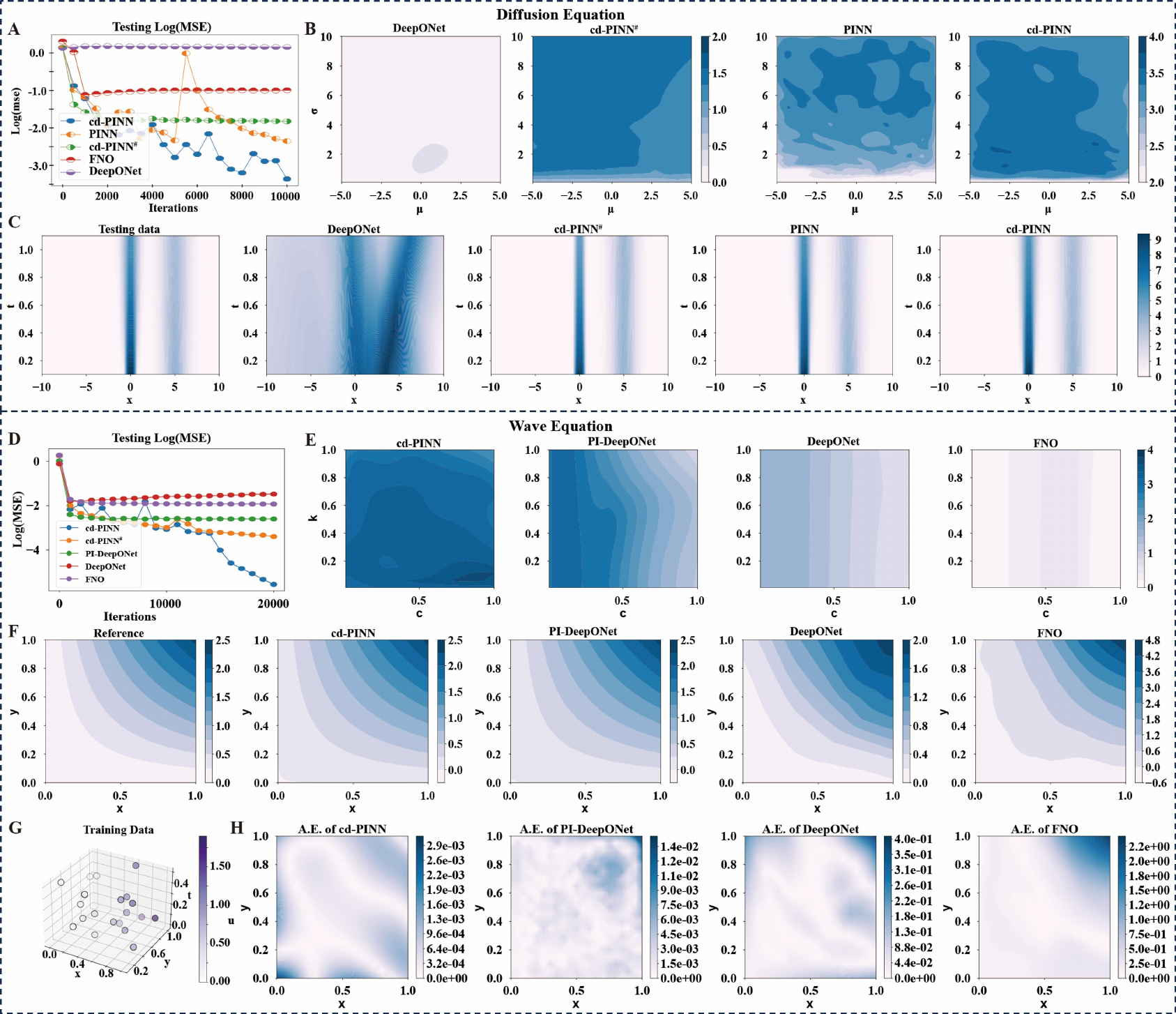}
    \caption{\textbf{Results of the parameterized diffusion and wave equations.} (\textbf{A}) The test MSE of cd-PINN, PINN, cd-PINN$^{\#}$, FNO, and DeepONet on the test dataset as the number of training epochs increases for the parameterized diffusion equation. (\textbf{B}) The NLMAE of predictions, (\textbf{C}) predictions on new configurations $\sigma=0.2, \mu=0.0$ without labeled training data, and (\textbf{D}) the test MSE as a function of the number of training epochs. (\textbf{E}) The NLMAE,   (\textbf{F}) predictions, and (\textbf{H}) absolute errors of cd-PINN, PI-DeepONet, DeepONet, and FNO for parameterized wave equation at $c=0.505, k=0.505$. (\textbf{G}) 20 labeled training data points randomly selected from the low-resolution with $c=0.505, k=0.505$.}
    \label{fig_Case_1}
\end{figure}

As shown in Fig.~\ref{fig_Case_1}D, after $20,000$ training epochs, cd-PINN$^{\#}$ exhibits a test MSE nearly one order of magnitude lower than DeepONet and FNO, while cd-PINN achieves a test MSE approximately two orders of magnitude lower than PI-DeepONet. Compared with FNO and DeepONet, cd-PINN and PI-DeepONet have significantly improved in NLMAE due to the addition of equation information (see  Fig. \ref{fig_Case_1}E). It is worth noting that FNO shows the worst performance, which might be attributed to the fact that FNO is suitable for learning the solution from parameters to a specific time moment, while other models can learn the solution mapping of the entire time region, in contrast. To align with other models, we also use time $t$ as the input to the FNO. Compared to PI-DeepONet, cd-PINN exhibits a notable advantage in NLMAE across a wide range of parameters and time regions, owing to the inclusion of constraints on the continuous dependence of the solution on the parameters. In addition, we also compared the two models in NRMSE, as shown in Table \ref{tab_diffusion_wave}, cd-PINN ($1.20\times 10^{-3}$) outperforms PI-DeepONet ($3.43\times 10^{-2}$) by an order of magnitude.

In addition to comparing the four models globally, we also made a comparison between the predicted solutions and true solutions of the four models under the parameter $c=0.505, k=0.505$, which is used for generating training data. Fig.~\ref{fig_Case_1}F gives a high-resolution true solution and the predicted solutions of cd-PINN, PI-DeepONet, DeepONet, and FNO from left to right, while their absolute errors are shown separately in Fig.~\ref{fig_Case_1}H. Again, FNO gives the worst prediction results. PI-DeepONet, cd-PINN, and DeepONet can all learn the overall shape of the solution, but in terms of the absolute error, DeepONet is much worse than both cd-PINN and PI-DeepONet. Moreover, cd-PINN has a significant improvement over PI-DeepONet in the maximum absolute error.

\subsection{Effect of Residual Point Density}
The density of residual points is a key factor influencing the performance of physics-informed models. As noted in \cite{yu2022gradient}, incorporating the equation gradient into the loss can enhance the model's accuracy even with a few residual points. The primary distinction between cd-PINN and PINN lies in the inclusion of the $\mathcal{L}_{cd}$ term. Here, we investigate how residual point density affects both models and assess the role of $\mathcal{L}_{cd}$ under low-density conditions. Consider a 2D Poisson equation, which is a representative of general elliptic equations, specified as
\begin{equation}
    u_{xx} + u_{yy} = -(a^2 + b^2) \cos(ax)\sin(by), \; x\in [0, \pi], \; y \in [0, \pi],
\end{equation}
with boundary condition
\begin{equation}
    \begin{split}
        u(x=0, y) &= \sin(by), \; y \in [0, \pi],\\
        u(x=\pi, y) &= \cos(a\pi)\sin(by), \; y \in [0, \pi],\\
        u(x, y=0) &= 0, \; x \in [0, \pi],\\
        u(x, y=\pi) &= \cos(ax)\sin(b\pi), \; x\in [0, \pi].
    \end{split}
\end{equation}
In practice, we use $20$ labeled data with $a=2.45, b=2.45$ (see Fig. \ref{fig_Case_2}E) and $2^{11}$ residual data points to train the model. The test data consists of $250$ sets of solutions corresponding to values of $a \in [0.0, 5.0]$ and $b \in [0.0, 5.0]$, totaling over 2 million data points.

In Fig.~\ref{fig_Case_2}A, the percentage stacked bar chart shows that cd-PINN achieves a lower NRMSE than PINN across different numbers of residual points, with the gap narrowing as the point density increases. Fig.~\ref{fig_Case_2}B illustrates that, with $2^{11}$ residual points, cd-PINN maintains a clear advantage in test MSE during training. This improvement may be attributed to the inclusion of $\mathcal{L}_{cd}$, which prevents the model from being trapped in a local minimum (Fig. \ref{fig_Case_2}C). The benefit of $\mathcal{L}_{cd}$ is also evident in Fig. \ref{fig_Case_2}D, where cd-PINN achieves higher NLMAE across a broad range of parameter configurations. The advantage becomes more pronounced for larger values of $a$ and $b$, corresponding to more complex, high-periodicity solutions. To illustrate this, we compare the true solution (Fig. \ref{fig_Case_2}F), and the predicted solutions of cd-PINN and PINN (Fig.\ref{fig_Case_2} G) when $a=5.0, b=5.0$. The latter leads to lower absolute errors by an order of magnitude (Fig.\ref{fig_Case_2}H).

\begin{figure}[H]
    \centering
    \includegraphics[width=1.0\linewidth]{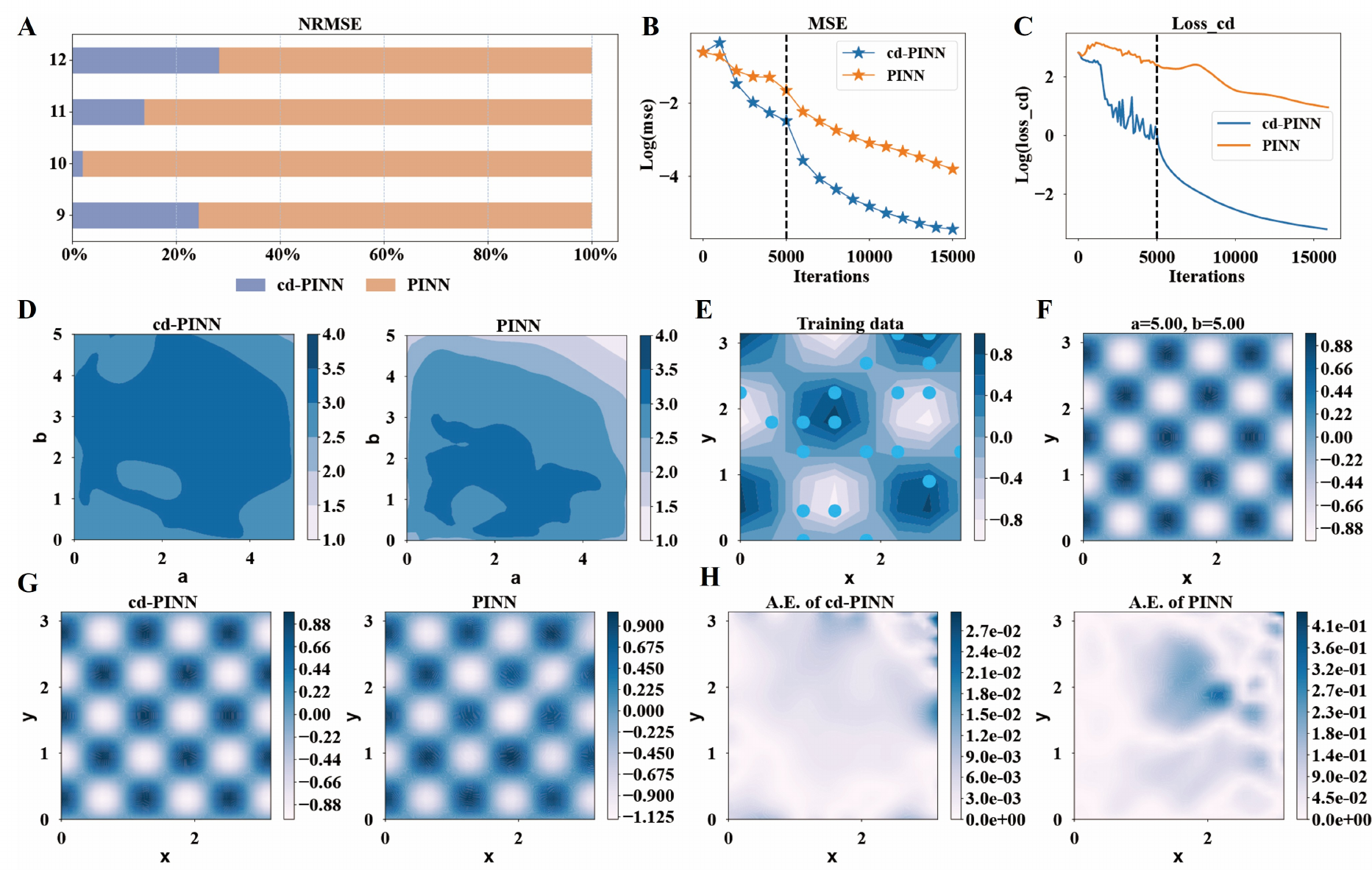}

    \caption{\textbf{Results of the parameterized Poisson equation.} (\textbf{A}) Comparison of the NRMSE between cd-PINN and PINN with respect to different numbers of residual points, where the vertical axis represents the number of $2^N$ residual data points used. (\textbf{B}) The test MSE and (\textbf{C}) $\mathcal{L}_{cd}$ term of cd-PINN and PINN as the number of training epochs changes when the number of residual points is fixed as $2^{11}$. (\textbf{D}) The NLMAE of predictions of cd-PINN and PINN. (\textbf{E}) 20 labeled training data points randomly selected from the low-resolution data when $a=2.45, b=2.45$. (\textbf{F}) The high-resolution true solution of the equation at new configurations $a=5.0, b=5.0$. (\textbf{G}) The predicted results and (\textbf{H}) absolute errors of cd-PINN and PINN at configurations $a=5.0, b=5.0$.}
    
    \label{fig_Case_2}
\end{figure}

\subsection{Influence of PDE Dimensionality}

One of the major strengths of physics-informed models is their ability to handle high-dimensional problems. Here, we investigate the impact of PDE dimensionality on the performance of cd-PINN and PINN using the d-dimensional diffusion-reaction equation with a negative linear term accounting for degradation.
\begin{equation}
    \frac{\partial u}{\partial t} = D\nabla^2u - \lambda u, x\in [0.0, 0.1]^{d}, t\in [0.0, 0.1],
\end{equation}
whose initial condition reads
\begin{equation}
    u(\boldsymbol{x}, 0) = \frac{1}{(2\pi \sigma^2)^{\frac{d}{2}}}\exp\Big(-\frac{\sum_{i=1}^{d}x_i^2}{2\sigma^2}\Big).
\end{equation}
For each $d=2,5,8$, we fixed $100$ labeled data with $D=0.06, \lambda=0.06$ and $2^{13}$ residual data points to train the model. The test data consists of $100$ sets of solutions corresponding to values of $D\in [0.01, 0.1]$ and $\lambda \in [-0.1, -0.01]$.

\begin{figure}[htbp]
    \centering
    \includegraphics[width=1.0\linewidth]{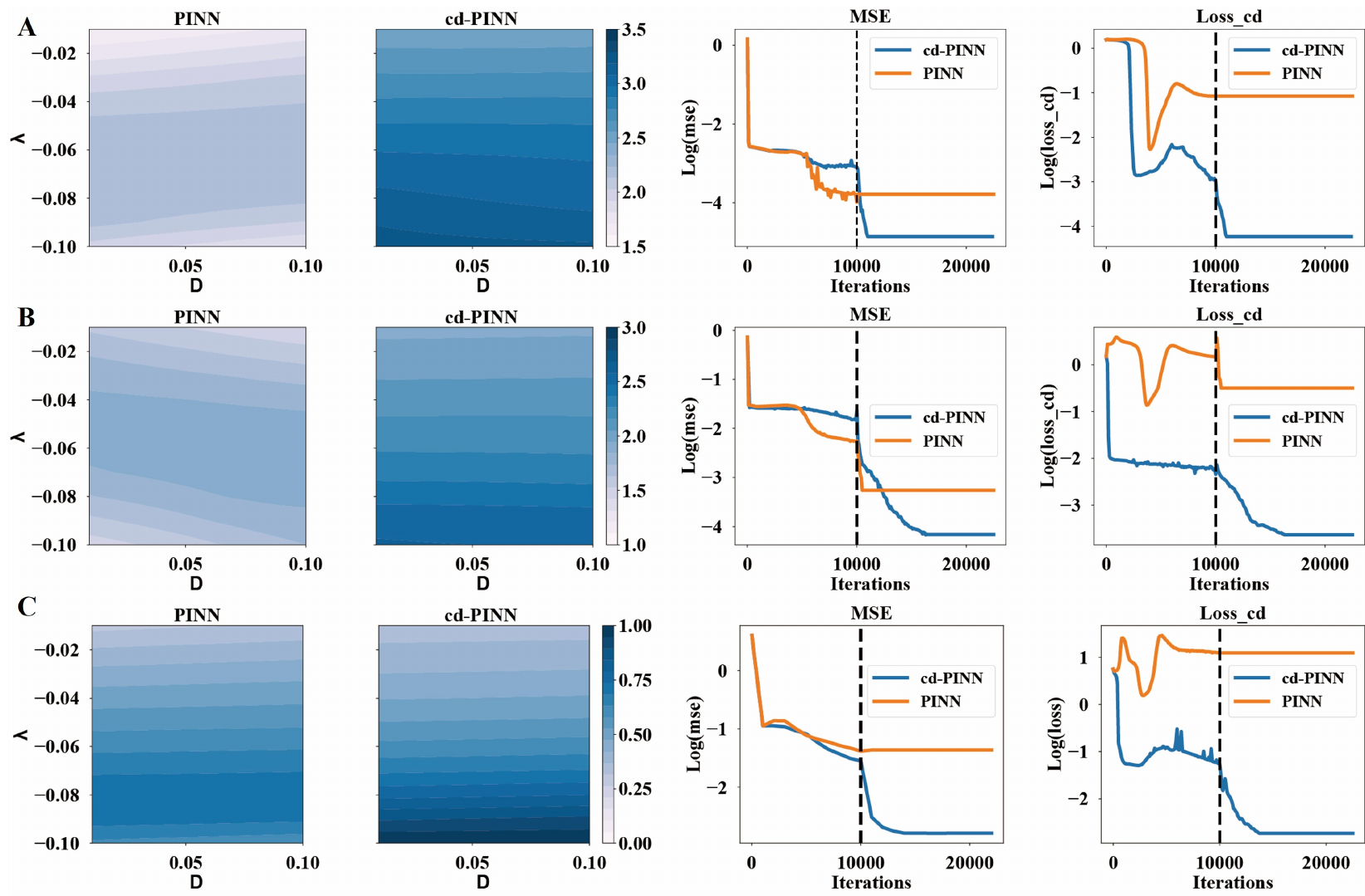}
    \caption{\textbf{Results of the parameterized diffusion-reaction equation.} \textbf{(A)}, \textbf{(B)}, and \textbf{(C)} show results for 2D, 5D, and 8D diffusion-reaction equations, respectively. For each dimension, from left to right: NLMAE of PINN and cd-PINN, test MSE, and $\mathcal{L}_{cd}$ versus training iterations.}
    \label{fig_Case_3}
\end{figure}

Fig.~\ref{fig_Case_3} A-C show the results of the cd-PINN and PINN models when applied to the diffusion-reaction equation in 2d, 5d, and 8d, respectively. The first and second columns display the NLMAE of cd-PINN and PINN, respectively. In all dimensions, the incorporation of the loss term $\mathcal{L}_{cd}$ results in a higher NLMAE for cd-PINN than that of PINN. However, since the number of residual points is fixed at $2^{13}$ across all cases, the model performance progressively degrades with the increase of dimension due to insufficient sampling of the domain. This highlights a major limitation of PINN-based models: although they can be extended to high-dimensional problems, a large number of residual points are required, and the GPU memory is extremely demanding. The third column illustrates how the test MSE of cd-PINN and PINN evolves over training iterations. The black dashed line represents the transition point between the epochs optimized by the Adam optimizer and those optimized by the LBFGS optimizer. As the dimensionality increases, the number of residual points shifts from sufficient to insufficient, resulting in an increasing performance gap between cd-PINN and PINN. The fourth column illustrates how the loss term $\mathcal{L}_{cd}$ changes during training for both models. Although $\mathcal{L}_{cd}$ is not included in the loss function of PINN, we compute and record it for comparison. In all dimensions, the value of $\mathcal{L}_{cd}$ in PINN converges to a higher level, and this convergence value increases with dimensionality.

\subsection{Comparison to FDM}

The Burgers equation is a widely used PDE in fluid mechanics, such as in the modeling of nonlinear waves and turbulence. As the viscosity coefficient goes to zero ($\nu\rightarrow0$), the solution develops shocks. To compare the computational efficiency and accuracy with traditional FDM, we consider the 1D Burgers equation
\begin{equation}
    \frac{\partial u}{\partial t} + u\frac{\partial u}{\partial x} = \nu \frac{\partial ^ 2 u}{\partial x^2}, \; x\in [-1.0, 1.0], \; t\in [0.0, 0.5],
\end{equation}
with initial and boundary conditions
\begin{equation}
    \begin{split}
        &u(x, 0) = -\sin(\pi x), \; x\in [-1.0, 1.0],\\
        &u(-1, t) = u(1, t)=0, \; t \in [0.0, 0.5].
    \end{split}
\end{equation}
We use this benchmark problem to evaluate the effectiveness and computational efficiency of the proposed cd-PINN model by direct comparison with the Newton-Implicit FDM. Specifically, we use 20 labeled data samples from $200\times 200$ numerical solution of Newton-Implicit FDM corresponding to $\nu=0.05$ and $2^{13}$ residual data points. The test data consist of $40$ solution instances corresponding to values of $\nu \in [0.01, 0.1]$, each with a resolution of $200\times 200$. Details of the Newton-Implicit FDM implementation and data generation process are provided in Supplementary Information.

To verify whether cd-PINN can effectively learn the solution of the equation as $\nu$ varies, we present the NLMAE of each pair $(t, \nu)$ in Fig.\ref{fig_Case_4} A. Overall, the mean absolute error of the predicted solution increases over time or as $\nu$ decreases. This aligns with our intuitions that the smaller $\nu$ is more prone to the formation of shocks. When $t>0.2$, the solution transits from a smooth profile to a viscous shock, making it increasingly difficult for the MLP-based model to fit. This explains why minimum NLMAE appears around $t\approx0.22, \nu=0.01$. In Fig.\ref{fig_Case_4} D, E, we compare the numerical solutions and predicted solutions of FDM and cd-PINN for $\nu = 0.01, 0.03, 0.05, 0.10$, respectively. The two solutions are visually almost indistinguishable. In Fig.~\ref{fig_Case_4}F, it can be seen that when $\nu=0.01$, the maximum absolute error reaches $0.2025$, and then gradually decreases as $\nu$ increases. In Fig.~\ref{fig_Case_4}G, we further present a comparison between the predicted solution of cd-PINN and the numerical solution obtained by FDM at different values of $\nu$ at time instances $t=0.0, 0.25$ and $0.50$. In addition to the cases within the residual point sample range $\nu \in [0.01, 0.10]$ -- specifically, $\nu=0.01, 0.05$ and $0.10$ -- we also include the results for the inviscid Burgers equation $(\nu=0.00)$, where the numerical solution is obtained using the Mac-Cormack method. Remarkably, the predicted solution of cd-PINN is almost visually indistinguishable from the numerical solution across all tested values of $\nu$, highlighting the strong generalization ability of the proposed model.

\begin{figure}[htbp]
    \centering
    \includegraphics[width=1.0\linewidth]{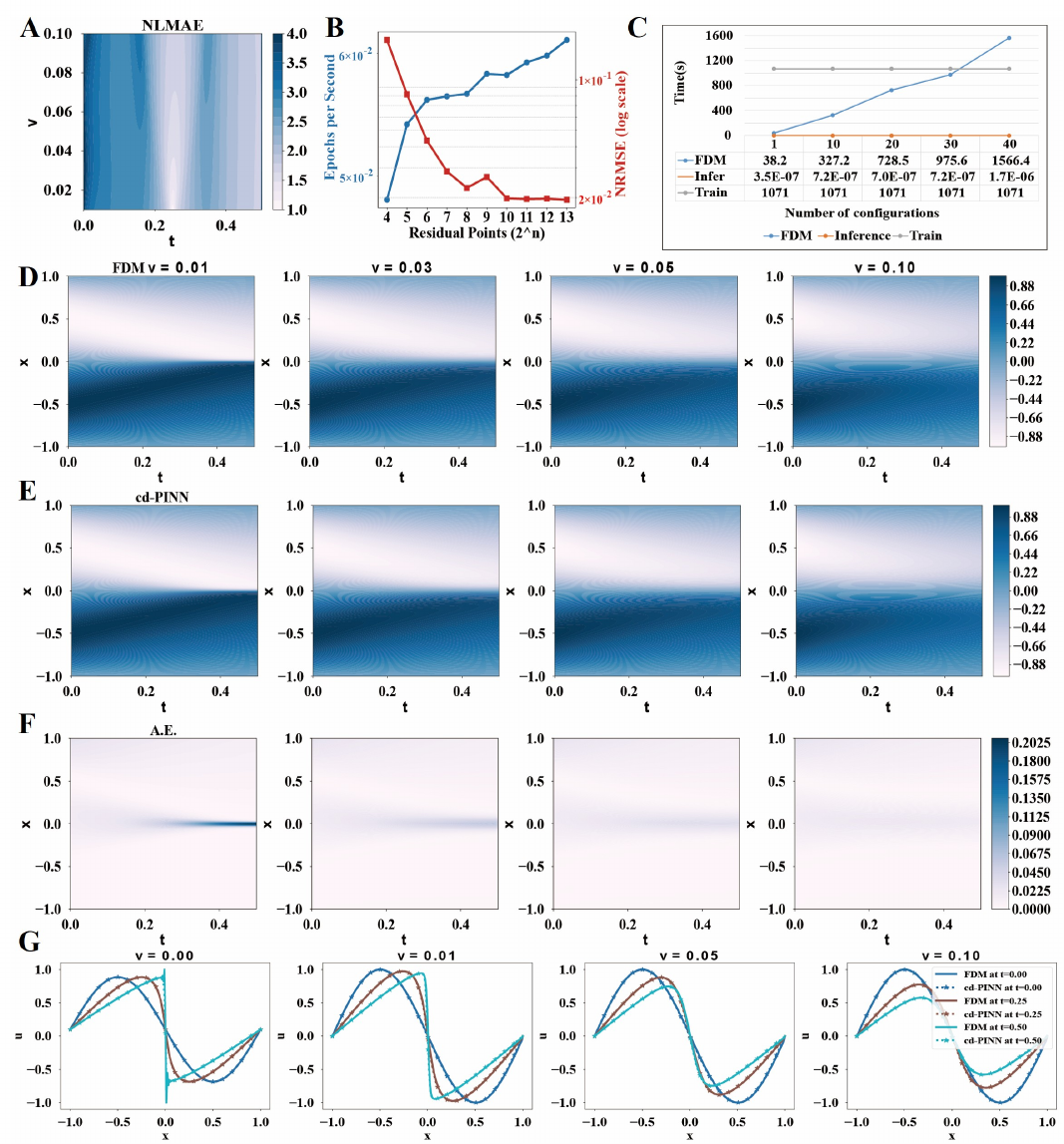}
    \caption{\textbf{Results of the parameterized Burgers equation.} \textbf{(A)} The NLMAE of predictions of the cd-PINN model under each set of $(t, \nu)$. \textbf{(B)} NRMSE and epochs per second $v.s.$ number of residual points. The blue line represents the epochs per second, while the red line represents the NRMSE of the predicted results. \textbf{(C)} Comparison on the computational efficiency of cd-PINN and Newton-Implicit FDM. \textbf{(D)} The numerical solutions of Newton-Implicit FDM, \textbf{(E)} predicted solutions of cd-PINN as well as \textbf{(F)} absolute errors between the two methods at $\nu = 0.01, 0.03, 0.05, 0.07$ and $0.10$, respectively. \textbf{(G)} Comparison of cd-PINN predicted solutions and numerical solution obtained via the FDM at $t=0.0, 0.25, 0.50$ for various values of $\nu$.}
    \label{fig_Case_4}
\end{figure}

In addition to comparing accuracy, we also evaluate the efficiency of the Newton-Implicit FDM and cd-PINN. The total number of residual points is a major factor determining the trade-off between model accuracy and computational cost. On one hand, more residual points increase the number of gradient computations, thereby extending the training time; on the other hand, they provide more comprehensive physical constraint information, potentially improving the accuracy. To investigate this trade-off, we examine how the number of residual points affects the time required for each training epoch and the final test NRMSE, while keeping other factors -- such as network depth/width and number of training iterations -- fixed. Fig.\ref{fig_Case_4} B shows how the total number of residual points impacts both epoch time and final NRMSE. As the number of residual points increases, the time per epoch increases gradually, while the final NRMSE decreases and then stabilizes. When the number of residual points reaches over $2^{13}$, we compare the training time, inference time, and cumulative time required for FDM to solve the equation for multiple $\nu$ values (Fig.\ref{fig_Case_4} C). Especially, when the number of $\nu$ values exceeds $30$, the cumulative computation time of FDM surpasses the total training time of cd-PINN, while the inference time of cd-PINN remains negligible.

\subsection{Applicability to Complex Systems}
Besides its accuracy, efficiency and transferability, cd-PINN has a remarkable capability to apply to complex systems and phenomena, which is demonstrated through the 2d Navier-Stokes equation \cite{stokes1851effect} for viscous, incompressible fluid in the vorticity form:
\begin{equation}
    \frac{\partial \omega}{\partial t} = -\boldsymbol{u}\nabla\omega + \nu\Delta \omega + f, \; \nabla \boldsymbol{u}=0,
\end{equation}
where $\boldsymbol{u}$ is the velocity field and $\omega = \nabla \times \boldsymbol{u}$ is the vorticity. $\boldsymbol{u}, \omega$ lie on a spatial domain with periodic boundary conditions, $\nu$ is the viscosity and $f$ is a constant forcing term. The spatial domain is $\Omega = [-0.5, 0.5)\times[-0.5, 0.5)$, the time interval is $t\in [0.0, 0.5]$, the viscosity is $\nu \in [10^{-3}, 10^{-2}]$, and the forcing term is set as:
\begin{equation}
    f(x_1, x_2) = 0.1\big(sin(2\pi(x_1+x_2))+cos(2\pi(x_1+x_2))\big),\; \forall x \in \Omega.
\end{equation}
Here we use 20 labeled data sampled from $128\times128$ numerical solution corresponding to $\nu=5e^{-3}$, and $2^{13}$ residual data points.

\begin{figure}[H]
    \centering
    \includegraphics[width=1.0\linewidth]{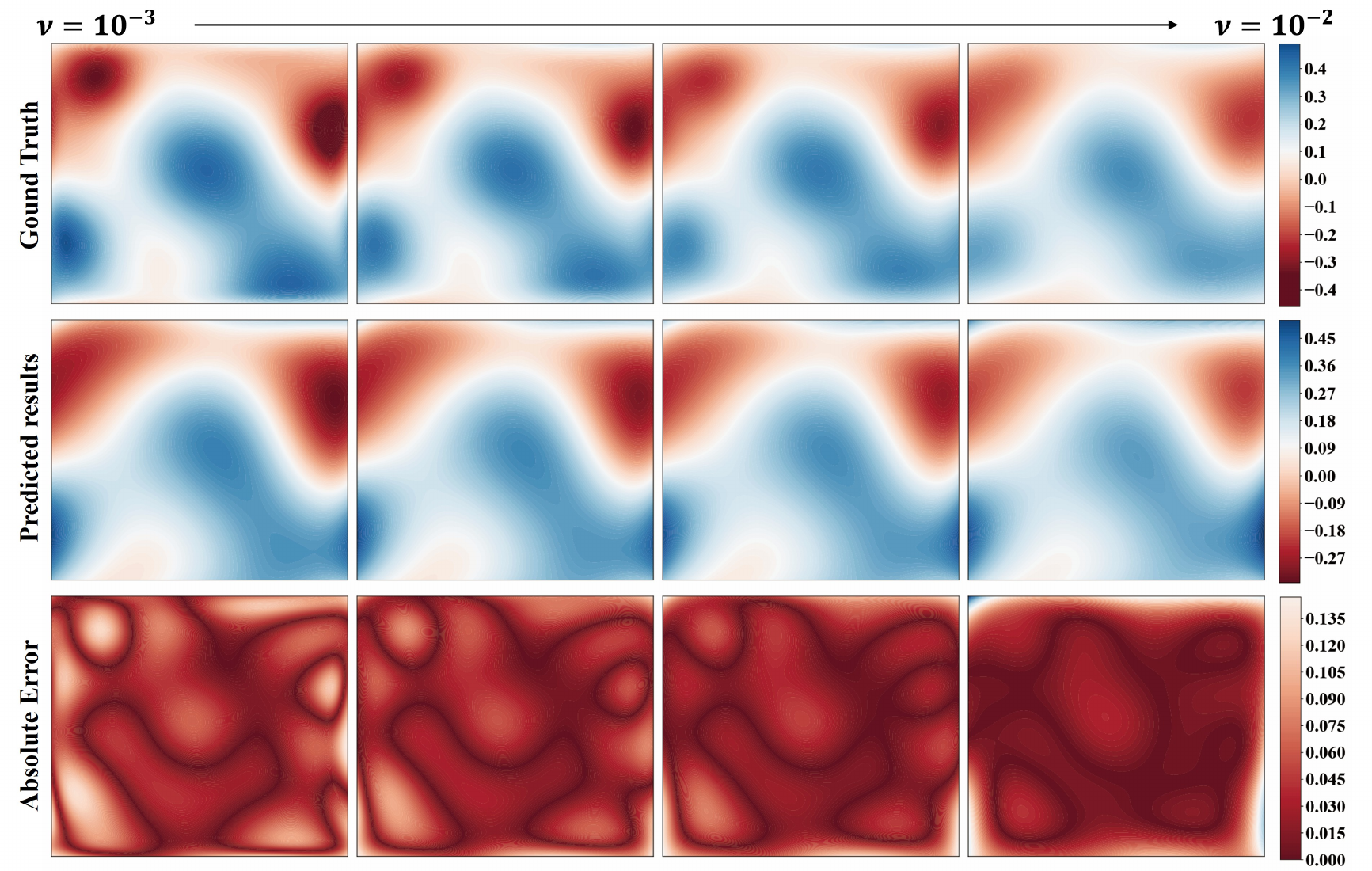}
    \caption{\textbf{Results of the parameterized Navier-Stokes equation.} The first row presents the reference solution computed using the spectral method. The second row displays the corresponding predictions obtained by the cd-PINN model, while the third row shows the point-wise absolute errors. From left to right, the viscosity coefficients are set to $1\times 10^{-3}, 3\times10^{-3}, 5\times10^{-3}$, and $1\times 10^{-2}$, respectively.}
    \label{fig_Case_5}
\end{figure}

This case demonstrates that the proposed cd-PINN is capable to capture the complex vortical behaviors of 2D viscous fluid, even though the model is trained using only a limited number of labeled data samples at a single viscosity coefficient. As shown in Fig. $\ref{fig_Case_5}$, the predictions of cd-PINN agree well with the reference solution in both low and high viscosity regions, and the point absolute error remains small throughout most of the spatial domain.

\subsection{Failure of Regularity}

Through previous examples, we demonstrate that our model performs well if the solution of the PDEs has a continuous dependence on initial values, parameters, and even source terms. A natural question arises: if there is no theoretical guarantee that the solution has a continuous dependence on the initial values or parameters, can cd-PINN still have an excellent generalization capability?

To address this critical issue, let us consider the initial-value problem of the diffusion equation with negative diffusion coefficients:
\begin{equation}
    \begin{split}
        & \frac{\partial u}{\partial t} - D\frac{\partial^2 u}{\partial x^2} = 0, \; (\boldsymbol{x}, t) \in Q_T,\\
        & u(\boldsymbol{x}, t) = u_0(\boldsymbol{x}), \; (\boldsymbol{x}, t) \in \partial_p Q_T,
    \end{split}
\end{equation}
where $Q_T = \Omega \times (0, T)$, and $\partial_p Q_T = \Omega \times \{t=0\}$, $\Omega \subset \mathbb{R}^n$ is a bounded region, $\partial \Omega \in C^{\infty}$, and $T > 0$. Perform the Fourier transform on the initial value $u_0(\boldsymbol{x})$, and we have $\hat{u}_0(\boldsymbol{k}) = \int_{-\infty}^{\infty}u_0(\boldsymbol{x})e^{-ik\boldsymbol{x}}d\boldsymbol{k}$, then the solution of the equation can be formally written as:
\begin{equation}
    u(\boldsymbol{x},t) = \frac{1}{2\pi}\int_{-\infty}^{\infty}\hat{u}_0(\boldsymbol{k})e^{-D\boldsymbol{k}^2t}e^{i\boldsymbol{k}x}d\boldsymbol{k}.
\end{equation}

It is easily seen that when $D>0$, the modal factor $e^{-D\boldsymbol{k}^2 t}$ decays with time $t$, and the high frequency decays faster, so that the solution of this problem remains smooth and stable. In contrast, when $D<0$, the model factor becomes $e^{|D|\boldsymbol{k}^2t}$, the high frequency component grows exponentially, meaning the solution will be exploded within a very small time, and the system becomes ill-posed. To be concrete, consider the one-dimensional case with $D=-1$ and the initial value $u_n(x,0) = e^{-n}\sin(nx)$\cite{hormander2007analysis, bers1964partial}. This problem can be explicitly solved with the solution $u_n(x,t) = e^{-n}\sin{(nx)}e^{n^2t}$. Then $u_n(x,0) \rightarrow 0$ uniformly as $n\rightarrow\infty$, so are all its spatial derivatives. But $\mathop{\sup}\limits_{x\in \mathbb{R}}|u_n(x,t)|\rightarrow\infty$ as $n\rightarrow\infty$ for any $t>0$. As a consequence, the solution does not continuously depend on the initial value.

Although the diffusion equation with negative diffusivity ($D<0$) is theoretically ill-posed and lacks continuous dependence on the initial value, the solution may still exhibit transient stability over a short time interval numerically. Consider the same initial value problem as in Section $\ref{sec_dif_wave}$, but choose $D=-0.002$. Within the prescribed time window, the solution remains well-defined and unique, with all other settings identical to those in Section $\ref{sec_dif_wave}$.

\begin{figure}[H]
    \centering
    \includegraphics[width=1.0\linewidth]{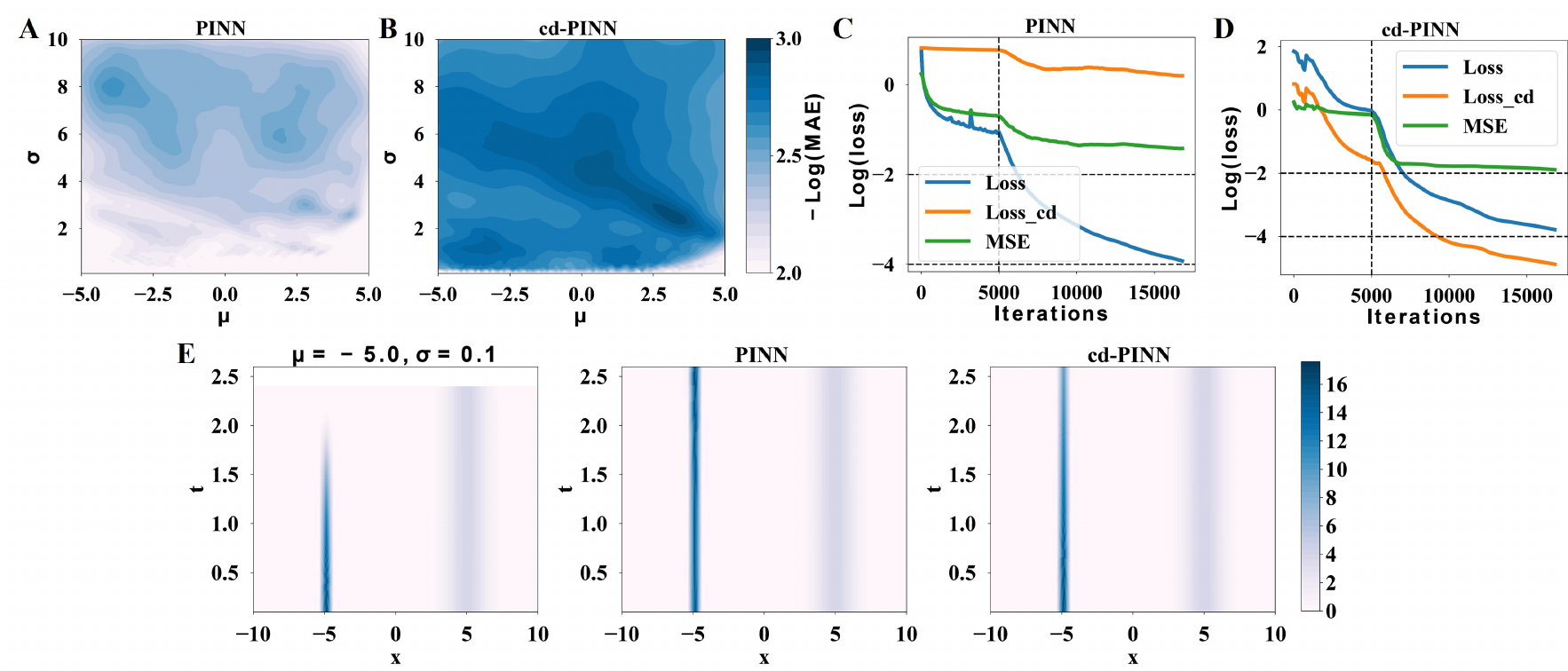}
    \caption{\textbf{Results of the parameterized diffusion equation with negative diffusion coefficients.} The NLMAE of predictions of (\textbf{A}) PINN and (\textbf{B}) cd-PINN under each configuration. The total loss $\mathcal{L}$, continuous dependence loss $\mathcal{L}_{cd}$, and the test MSE of (\textbf{C}) PINN and (\textbf{D}) cd-PINN as the number of training epochs changes. (\textbf{E}) The true solution (left), and predicted solutions by PINN (middle) and cd-PINN (right) at $\mu=-5.0, \sigma=0.1$.}
    \label{fig_Case_6}
\end{figure}

As shown in Fig.\ref{fig_Case_6} A and B, although the solution is only stable within a limited time window, both PINN and cd-PINN demonstrate good performance in prediction. The inclusion of $\mathcal{L}_{cd}$ enables cd-PINN to outperform PINN, achieving a prediction NRMSE of $6.65\times10^{-2}$ compared to $1.23\times 10^{-1}$ for PINN. Subfigures C and D illustrate the decrease of the total loss, $\mathcal{L}_{cd}$, and the test MSE over training epochs for PINN and cd-PINN, respectively. Unlike the case of $D>0$ in Section $\ref{sec_dif_wave}$, the MSE does not converge to the same magnitude as the total loss, indicating a certain degree of overfitting. This may be attributed to insufficient regularization of both the PDE loss and the differentiability loss under the ill-posed condition when $D < 0$. To examine model performance under blow-up conditions, we visualize the true and predicted solutions of PINN and cd-PINN for $\mu=-5.0,$ and $\sigma=0.1$ in Fig.\ref{fig_Case_6} E. The solution diverges to infinity for $t\geq2.5$, and thus the true solution is omitted beyond this point. While neither model accurately captures the blow-up, cd-PINN provides a prediction close to the true solution in the pre-blow-up regime.

%%%%%%%%%%%%%%%%%%%%%%%%%%%%%%
\subsection{Real Application: Protein Aggregation Dynamics in Alzheimer's Disease}
Alzheimer's disease (AD) is a progressive neurodegenerative disorder characterized by the accumulation and spread of misfolded proteins, mainly amyloid $\beta$-protein (A$\beta$) plaques and hyperphosphorylated tau protein tangles, within interconnected brain regions. The spatiotemporal evolution of protein concentration $c\left(\mathbf{x},t\right)$ follows the classical Fisher-Kolmogorov (F-K) equation:
\begin{equation}
\frac{\partial c}{\partial t} = D \nabla^2 c + \alpha \, c(1-c),
\label{eq:fisher_kpp_AD}
\end{equation}
where $c(\mathbf{x}, t) \in [0,1]$ represents the normalized protein concentration, $D$ is the diffusion coefficient, and $\alpha$ is the aggregation rate. We aim to estimate unknown parameters $D,\alpha$ from few-shot observations.

We adopt an offline-online strategy for this inverse problem, with the complete 
algorithm detailed in Algorithm~S1 (see Supplementary Information).  During the offline stage, we first train cd-PINN to approximate the solution $c(t, x; D, \alpha)$ over a wide range of reasonable parameter values. Next, the online stage freezes the network weights and, given observations $\{c_{\text{obs}}(t_i, x_i)\}_{i=1}^{N_{\text{obs}}}$, estimates $D$ and $\alpha$ by directly solving a least-squares minimization problem
\begin{equation}
    (D^*, \alpha^*) = \argmin_{D, \alpha} \frac{1}{N_{\mathrm{obs}}} \sum_{i=1}^{N_{\mathrm{obs}}} \left| \hat{c}(x_i, t_i; D, \alpha) - c_{\mathrm{obs}}^{(i)} \right|^2,
    \label{eq:inverse_1d}
\end{equation}
where $\hat{c}$ denotes predictions of the pre-trained network with frozen weights. The explicit encoding of parameters $(D, \alpha)$ as network inputs enables gradient-based optimization. Compared to the traditional way of solving inverse problems by using PINNs, our current approach avoids repeatedly re-tuning neural networks for finding a sufficiently good approximation to the PDE solution during parameter updates, and thus saves a huge amount of time which is unaffordable for online tasks. 

We first examine our approach with respect to the continuous F-K equation \eqref{eq:fisher_kpp_AD} on synthetic data set. Both PINN and cd-PINN are trained at a single parameter configuration $(D=0.3, \alpha=1.0)$ and evaluated across a broad parameter range. As shown in Figs.~\ref{fig_8}A--D, cd-PINN significantly reduces the prediction error throughout the entire parameter space, and shows remarkable improvement in the regions far from the training point. The test MSE convergence curves (Fig.~\ref{fig_8}E) reveal that cd-PINN not only converges faster but also achieves lower final errors ($3.42 \times 10^{-5}$) than PINN ($1.04 \times 10^{-3}$). For the inverse problem, we evaluate parameter estimation performance across all 99 test configurations (Fig.~\ref{fig_8}H--I). cd-PINN outperforms PINN in 73 cases (73.7\%), achieving an average parameter estimation error of 4.74\% compared to 8.23\% for PINN. To understand this improvement, we visualize the loss landscapes and optimization trajectories for a representative case with true parameters $(D=0.15, \alpha=1.30)$ and initial guess $(D=0.40, \alpha=0.70)$ in Fig.~\ref{fig_8}F--G. The loss landscape of cd-PINN exhibits a smoother, more convex structure that facilitates gradient-based optimization, with the trajectory converging to $(D=0.130, \alpha=1.299)$, close to the ground truth. In contrast, the PINN's landscape contains spurious local minima, and the optimizer converges to $(D=0.100, \alpha=1.255)$, a less accurate estimate.
\begin{figure}[H]
\centering
\includegraphics[width=1.0\linewidth]{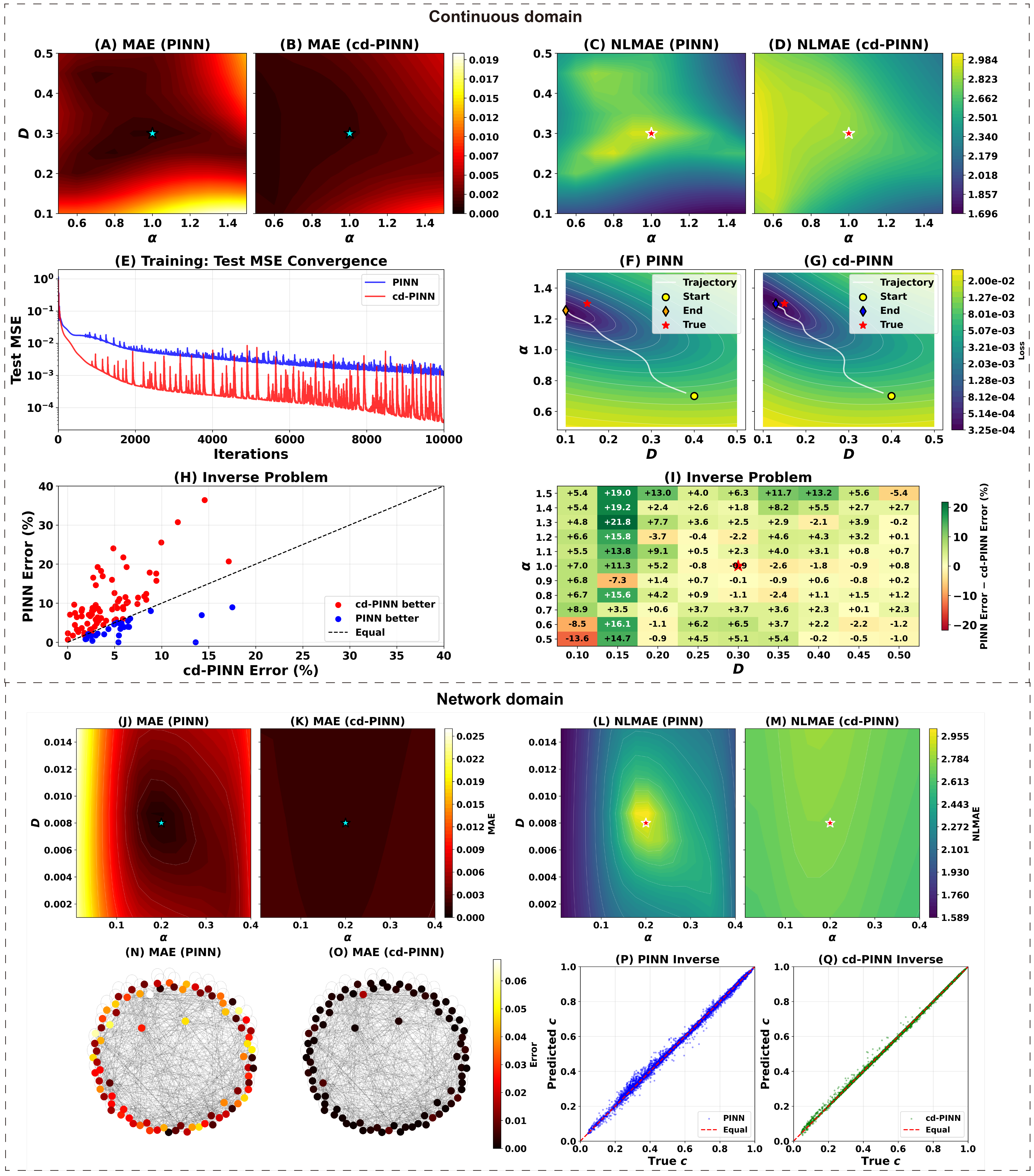}
\caption{\textbf{Evaluation of PINN and cd-PINN for Fisher-Kolmogorov equation in continuous and network settings.} \textit{Continuous domain (synthetic data):} 
    (\textbf{A}--\textbf{B}) MAE and (\textbf{C}--\textbf{D}) NLMAE across diffusion-aggregation parameter space $(D, \alpha)$, with models trained at single configuration. 
    (\textbf{E}) Test MSE convergence during training over 10,000 iterations. 
    (\textbf{F}--\textbf{G})  Loss landscapes for the inverse problem with true parameters $D=0.15$ and $\alpha=1.3$.
   ( \textbf{H}) Comparison of parameter estimation errors across all 99 test configurations. \textbf{(I)} Heatmap of error differences in the $(D, \alpha)$ space.
    \textit{Network domain (brain connectome):} 
   (\textbf{J--K}) MAE and \textbf{(L--M)} NLMAE heatmaps of PINN and cd-PINN.
   (\textbf{N--O}) Spatial distribution of prediction errors on the brain connectivity graph. Node colors represent absolute errors at each brain region.
    (\textbf{P}--\textbf{Q}) Scatter plots of true and predicted solutions for inverse parameter estimation across all test configurations.}
\label{fig_8}
\end{figure}

%%Fisher-KPP network
The human brain exhibits a complex connectome structure, within which misfolded proteins spread over anatomically connected regions. To account for the network topology, we discretize the continuous spatial domain onto the brain connectome, representing it as a graph with $N$ cortical and subcortical regions. The continuous Laplacian operator $\nabla^2$ in the F-K model is then replaced by the graph Laplacian $\mathbf{L} = \mathbf{D}_{\text{deg}} - \mathbf{A}$, where $\mathbf{A}$ is the adjacency matrix encoding structural connectivity of different brain regions, and $\mathbf{D}_{\text{deg}}$ is the degree matrix. The discretized F-K equation on networks reads:
\begin{equation}
\frac{d\mathbf{c}}{dt} = -D \mathbf{L} \mathbf{c} + \alpha \mathbf{c} \odot (1 - \mathbf{c}), \quad t \in [0, T],
\label{eq:graph_fisher_kpp}
\end{equation}
where $\mathbf{c} \in \mathbb{R}^N$ represents protein concentrations in $N$ brain regions, $\odot$ denotes element-wise multiplication, $D > 0$ is a constant diffusion coefficient, and $\alpha > 0$ is the common aggregation rate. 

We evaluate the model performance on both synthetic data generated by the discretized F-K equation on a brain network with $N=83$ regions and real medical data. For the forward problem, cd-PINN achieves a test MSE of $2.93 \times 10^{-5}$ compared to $1.81 \times 10^{-4}$ for PINN. Beyond global accuracy metrics, we also examine the spatial distribution of prediction errors on the brain connectivity graph (Fig.~\ref{fig_8}N--O). The continuous-dependence requirement enables the cd-PINN to achieve lower spatial errors in all 83 brain regions. For the inverse problem, we tested both models across all parameter combinations with observations added by 2\% Gaussian noise. Supplementary Figs.~7 and 8 show detailed temporal predictions for three representative regions of the brain (nodes 0, 40, and 82) in all combinations of parameters. When estimating parameters from simulated observations, the RMSE of cd-PINN across the entire parameter space is lower than that of PINN (Fig.~\ref{fig_8}P--Q). Especially under low and high concentration conditions, its performance improvement is significant. 

To further verify the practical applicability of our method, we then performed patient-specific parameter estimation using real clinical data from the Alzheimer's Disease Neuroimaging Initiative (ADNI)\cite{petersen2010alzheimer}. Specifically, we use the ADNI UC Berkeley tau partial volume–corrected (PVC) dataset, which provides continuous standardized uptake value ratios (SUVRs) quantifying tau PET uptake, together with FreeSurfer-defined regional volumes (in $\text{mm}^3$) for each PET region of interest (ROI). In our experiments, we selected subjects with data available for at least four consecutive years from this dataset.

Alzheimer's disease typically progresses through various clinical stages, ranging from individuals with normal cognition (CN) to those with early mild cognitive impairment (EMCI), mild cognitive impairment (MCI), late mild cognitive impairment (LMCI), and finally patients diagnosed with Alzheimer's disease (AD). Fig.~\ref{fig_9}B presents the longitudinal SUVR prediction results for six representative patients at different disease stages. Both PINN and cd-PINN successfully fit the observed data within the measurement period, with the fitting error of cd-PINN is always lower than that of PINN. The shaded area indicates the uncertainty of the model (meaning $\pm 1$ standard deviation), while cd-PINN consistently produces more stable extrapolations with tighter confidence bounds. Beyond temporal dynamics, we also examined the distribution of tau protein in different brain regions. Fig.~\ref{fig_9} C displays brain surface SUVR maps for both CN and AD patients at different time points. The upper rows show the observed PET data, while the lower rows display the cd-PINN fitted and predicted values. For the CN patient, throughout the entire observation period as well as the predicted future stages (indicated by question marks), the accumulation of tau protein in most areas of the brain is extremely low. In contrast, the AD patient exhibits substantially higher and more widespread tau pathology. Additional brain surface visualizations for intermediate disease stages (EMCI, MCI, and LMCI patients) are provided in the Supplementary Information.

\begin{figure}
\centering
\includegraphics[width=1\linewidth]{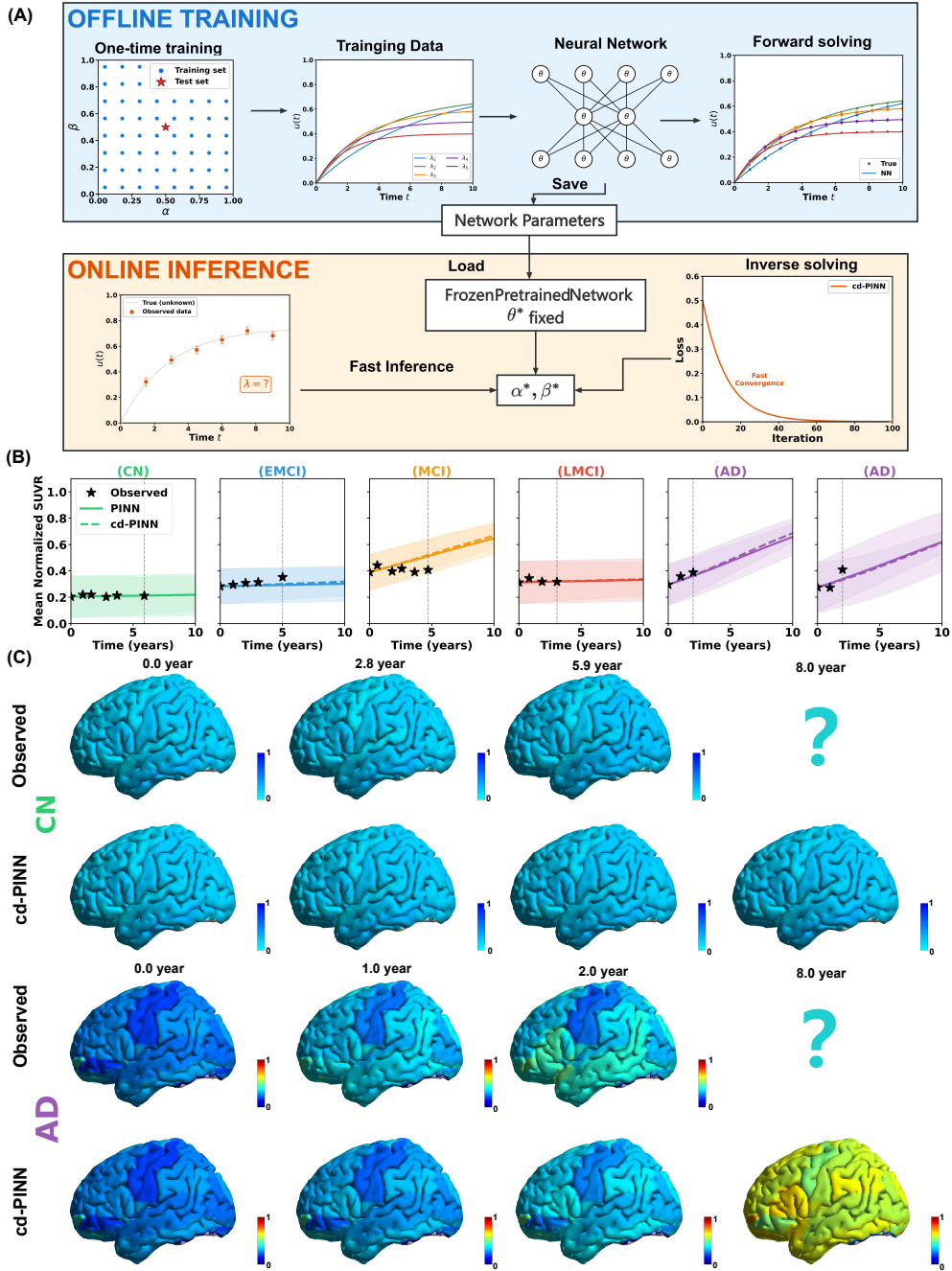}
\caption{\textbf{Patient-specific modeling of tau protein aggregation using real clinical data.} \textbf{(A)}Schematic of the cd-PINN framework. \textit{Top (Offline Training):} Neural network are pre-trained over a wide range of parameters, whose weights $\theta^*$ are saved for online tasks. \textit{Bottom (Online Inference):} Frozen pretrained network enables fast estimation of $\alpha^*, \beta^*$ from new observations without retraining. 
\textbf{(B)} Longitudinal SUVR predictions for patients across different disease stages (CN, EMCI, MCI, LMCI, AD). Stars: observed data; solid/dashed lines: PINN/cd-PINN predictions; shaded regions: $\pm$1 standard deviation; vertical lines: last observation time.
\textbf{(C)} Brain surface SUVR maps comparing CN (top) and AD (bottom) patients. Upper rows: observed PET data; lower rows: cd-PINN fitted/predicted values. ``?'' indicates unavailable future observations.} 
\label{fig_9}
\end{figure}

%%%%%%%%%%%%%%%%%%%%%%%%%%%%%%

\section{Methods}\label{sec_methods}
We consider the dynamical system in the form of
\begin{equation}
    \begin{split}
        \frac{du}{dt} & = \mathcal{P}(u, a), \hspace{1.5cm} in \; \Omega\times [0, T],\\
        u &= g, \; \hspace{2.3cm} in \; \partial \Omega \times [0, T],\\
        u & = u_0, \; \hspace{2.1cm} in\; \bar{\Omega}\times \{0\}.
    \end{split}
    \label{pde}
\end{equation}
where $\Omega \subset \mathbb{R}^d$ is a bounded, open set, the vector $a \in \mathcal{A}\subset\mathbb{R}^{d_a}$ denotes the PDE coefficients (or parameters), and the function $g$ gives a fixed boundary condition, which can also potentially be entered as a parameter. $u(t, \boldsymbol{x})\in \mathcal{U}=\mathcal{U}(\Omega)$ is the unknown for each fixed time point $t\geq0$, where $\mathcal{U}$ is a properly defined function space for the PDE solution. $u_0\in\mathcal{U}$ denotes the initial condition.

\subsection{Regularity Theory for Parabolic PDE}

Here, we take the Cauchy problem of the diffusion equation as an illustration\cite{evans2022partial}.
\begin{equation}
    \begin{split}
        & \frac{\partial u}{\partial t} - \Delta u = f(\boldsymbol{x}, t), \; (\boldsymbol{x}, t) \in Q_T,\\
        & u(\boldsymbol{x}, t) = u_0(\boldsymbol{x}, t), \; (\boldsymbol{x}, t) \in \partial_p Q_T,
    \end{split}
    \label{eu_diffusion}
\end{equation}
For the elliptic-type Poisson equation and the hyperbolic-type wave equation, there exist similar results; see Supplementary Information for details. 

\begin{theorem}{\textbf{(Maximum Principle for Cauchy Problem of Diffusion Equation)}}\label{theorem_max_heat}
    Suppose $u \in C^{2,1}(\Omega \times (0, T]) \cap C(\Omega \times [0, T])$ solves
    \begin{equation}
        \begin{split}
            & \frac{\partial u}{\partial t} - \Delta u = 0, \; (\boldsymbol{x},t) \in Q_T,\\
            & u(\boldsymbol{x}, t) = u_0(\boldsymbol{x}), \; (\boldsymbol{x}, t) \in \partial_p Q_T.
        \end{split}
        \label{heat_nonsource}
    \end{equation}
    and satisfies the growth estimate
    \begin{equation}
        u(\boldsymbol{x}, t) \leq Ae^{a|\boldsymbol{x}|^2}, \; (\forall\boldsymbol{x} \in \Omega, \; 0 \leq t\leq T)
    \end{equation}
    for constants $A$, $a>0$. Then
    \begin{equation}
        \mathop{sup}\limits_{\Omega\times [0,T]}u(\boldsymbol{x}, t) = \mathop{sup}\limits_{\Omega}\; \;  u_0(\boldsymbol{x}),
    \end{equation}
    where $Q_T = \Omega \times (0, T)$, and $\partial_p Q_T = \Omega \times \{t=0\}$, $\Omega \subset \mathbb{R}^n$ is a bounded region, $\partial \Omega \in C^{\infty}$, and $T > 0$.
\end{theorem}

According to Theorem \ref{theorem_max_heat} and some simple proof, we have the following theorem.

\begin{theorem}{\textbf{(Existence and Uniqueness of Diffusion Equation)}}\label{theorem_eu_heat}
    Let $u_0 \in C(\Omega), \; f\in C(\Omega\times [0, T])$. Then there exists at most one solution $u \in C^{2,1}(\Omega \times (0, T])\cap C(\Omega \times [0, T])$ to the Cauchy problem $(\ref{eu_diffusion})$ satisfying the growth estimate $|u(\boldsymbol{x}, t)|\leq Ae^{a|\boldsymbol{x}|^2}, \; (\forall\boldsymbol{x} \in \Omega, 0\leq t\leq T)$ for constants $A, a > 0$.
\end{theorem}

In Theorem $\ref{theorem_eu_heat}$, we show the existence and uniqueness of the solution to the diffusion equation, which is bounded in $\Omega \in \mathbb{R}^n$ and the boundary $\partial \Omega \in C^{\infty}$, making it not applicable to general regions. However, we can relax the restrictions on the region as well as the requirements of $f$ and $u_0$. More details can be found in Supplementary Information. 

Suppose $u_1$ and $u_2$ are the solutions to equation $(\ref{eu_diffusion})$ with the same source term $f$ but different initial conditions $u_{0_1}$ and $u_{0_2}$, satisfying $|u_{0_1} - u_{0_2}| < \epsilon$ everywhere on $\partial_p Q_T$, where $\epsilon\ll1$ is a given constant. Let $w = u_1 - u_2$, it is easy to prove that $w$ is the solution of diffusion equation in $(\ref{heat_nonsource})$ with the initial condition $u_{0_1}-u_{0_2}$. According to Theorem $\ref{theorem_max_heat}$,
\begin{equation}
    \mathop{sup}\limits_{\Omega\times [0, T]} \; \; \omega(\boldsymbol{x}, t) = \mathop{sup}\limits_{\Omega}\; \; (u_{0_1}(\boldsymbol{x})- u_{0_2}(\boldsymbol{x})) \leq \epsilon,
\end{equation}
and $\omega$ is positive on $\Omega \times [0, T]$, we have $|\omega| = |u_1 - u_2| \leq \epsilon$, which means the solution of diffusion equation in $(\ref{eu_diffusion})$ is continuously depending on the given initial condition. 

On the other hand, towards the diffusion coefficient $D$ in the diffusion equation, we consider the following simple situation. Let the diffusion constants $D_1$ and $D_2$ be greater than $0$, satisfying $|D_1 - D_2| < \epsilon$, where $\epsilon\ll1$ is a given positive constant. Let $u_1, u_2 \in C^{\infty}(\Omega)$ satisfy
\begin{equation}
    \begin{split}
        & \frac{\partial u_1}{\partial t} - D_1 \Delta u_1 = f(\boldsymbol{x}, t), \; (\boldsymbol{x}, t) \in \Omega\times (0, T),\\
        & \frac{\partial u_2}{\partial t} - D_2 \Delta u_2 = f(\boldsymbol{x}, t), \; (\boldsymbol{x}, t) \in \Omega\times (0, T),\\
        & u_1 = u_2 = u_0(\boldsymbol{x},t), \;(\boldsymbol{x},t) \in \Omega \times \{t=0\},\\
        & u_1 = u_2 = \varphi(\boldsymbol{x}, t), \; (\boldsymbol{x}, t) \in \partial \Omega \times (0, T).
    \end{split}
\end{equation}
Let $D_1 = D_2 + r \; (r<\epsilon)$, and $\omega = u_1 - u_2$, then $\omega$ satisfies
\begin{equation}
    \begin{split}
        & \frac{\partial \omega}{\partial t} - D_2 \Delta \omega = r\Delta u_1, \; (\boldsymbol{x}, t) \in \Omega \times (0, T),\\
        & \omega = 0, \; (\boldsymbol{x}, t) \in \Omega \times \{t=0\},\\
        & \omega = 0, \; (\boldsymbol{x}, t) \in \partial \Omega \times (0, T).
    \end{split}
\end{equation}
That is, $\omega$ satisfies a non-homogeneous diffusion equation with an initial boundary value $0$ and a source term $r\Delta u_1$.

\begin{theorem}{\textbf{(Energy Estimation for Second-Order Parabolic Equations)}}\label{energy_parabolic}
    For the general second-order parabolic equation with initial and boundary conditions as
    \begin{equation} 
    \begin{split} 
    &\frac{\partial u}{\partial t} - \Big(\mathop{\sum}\limits_{i,j=1}\frac{\partial}{\partial \boldsymbol{x}_i}\Big(a_{ij}( \boldsymbol{x}, t) \frac{\partial u}{\partial \boldsymbol{x}_j}\Big) + \sum_{i=1}^{n}b_{i}( \boldsymbol{x},t)\frac{\partial u}{\partial \boldsymbol{x}_i} + c(\boldsymbol{x}, t)u\Big) = f( \boldsymbol{x}, t),\\
            & u (\boldsymbol{x}, t) = g(\boldsymbol{x}, t), \; (\boldsymbol{x}, t) \in \Omega \times \{t=0\},\\
            &u(\boldsymbol{x}, t) = 0, \; (\boldsymbol{x}, t) \in \partial \Omega \times (0, T),
        \end{split}
    \label{eq_parabolic}
    \end{equation}
    where $(a_{ij}(\boldsymbol{x},t))_{i,j=1,\cdots, n}$ is a positive-definite matrix, that is for a constant $\alpha > 0$ it satisfies
    \begin{equation}
        \mathop{\sum}\limits_{i,j=1}^{n} a_{ij}\xi_i \xi_j \geq \alpha |\xi|^2, \; \forall\xi \in \mathbb{R}^n.
    \end{equation}
    The coefficients on the left side of equation $(\ref{eq_parabolic})$ are assumed to be $C^{\infty}$ functions in the region under consideration. 
    
    Define $Q_T = \Omega \times (0, T)$. Assume $\boldsymbol{u} \in C^{\infty}(\bar{Q}_T)$ is the solution of equation $(\ref{eq_parabolic})$, and $f \in L^2(Q_T)$, then $\forall t \in [0,T]$, we have the estimates
    \begin{equation}
        E(t) \leq C\big[E(0) + \int_{0}^T\int_{\Omega}f^2 d\boldsymbol{x}dt\big],
    \end{equation}
where $E(t)=\int_{\Omega} u^2 (\boldsymbol{x}, t) d\boldsymbol{x}$ denotes the energy.
\end{theorem}

It is easy to verify that $\omega$ satisfies the equation $(\ref{eq_parabolic})$ and the relevant conditions in Theorem $\ref{energy_parabolic}$, so we get
\begin{equation}
    \int_{\Omega} \omega^2(\boldsymbol{x}, t)d\boldsymbol{x} \leq Cr^2 \int_{0}^T\int_{\Omega} \Delta u_1^2d\boldsymbol{x}dt < \tilde{\epsilon}.
\end{equation}
It follows that in this case, the solution to the diffusion equation is continuously dependent on the diffusion constant. Regarding the source term, we can also draw similar conclusions based on the energy estimation inequality above.

Towards the parameterized PDE solution problem in \eqref{pde}, a central premise for an algorithm's success, no matter it is an operator learning or the improved PINN algorithm based on meta-learning, is that the PDE solution is existed and unique under the new parameters and initial/boundary values. Therefore, in this paper, we assume that for any set of configurations in the PDE to be solved, the classical solution exists uniquely without proof.

\subsection{PINN with Continuous Dependence (cd-PINN)}

We assume that the classical solution $u$ to a PDE system uniquely exists and is bounded for all time and for every $u_0 \in \mathcal{U}$. Let
\begin{equation}
    G: \mathcal{A} \times \mathcal{U} \times [0,T] \times \Omega \mapsto \mathbb{R}^{d_u}, \quad G(a, u_0, t, \boldsymbol{x}) = u(t, \boldsymbol{x}), \quad  \forall(t, \boldsymbol{x}) \in [0, T] \times \Omega.
\end{equation}
be a nonlinear map. We study maps $G$ which arise as the solution operators of the parametric PDEs. Suppose we have a few observations $\{a_j, u_{0_j}, u_j\}_{j=1}^{N}$, where $u_j=u(a_j, u_{0_j}, t, \boldsymbol{x}) = G(a_j, u_{0_j}, t, \boldsymbol{x})$, seen Fig. \ref{fig_diagram} D. In practice, $N$ is often a very small number. In some of our examples, we choose $N=1$.

Now, the question is how to identify and learn the solution to the PDE among multiple configurations. For operator learning methods such as FNO and DeepONet, they build a specific model architecture and systematically learn the manifold architecture of solution mapping. For example, DeepONet aims to learn the mapping function from parameters to solutions: $G: u_0 \mapsto u$, see Fig.~\ref{fig_diagram}A. For meta-learning based PINN Meta-Auto-Decoder(MAD) \cite{huang2022meta}, it tries to learn a Lipschitz continuous mapping in a low-dimensional space: $\bar{G}: \boldsymbol{c}\mapsto u$, where $\boldsymbol{c}\in C\subset\mathbb{R}^{l}(l \ll d_a)$, such that $G(\mathcal{A}) \subset \bar{G}(C)$, see Fig.\ref{fig_diagram} B. In other words, for any $u_0 \in \mathcal{A}$, it hopes that there exists $\boldsymbol{c} \in C$ satisfying $\bar{G}(\boldsymbol{c}) = G(u_0)$.

Our goal is to learn an operator $G: (\boldsymbol{a},\boldsymbol{u}_0,t,\boldsymbol{x}) \mapsto u(t,\boldsymbol{x})$, where $u(t, \boldsymbol{x})$ satisfies the parameterized differential equation ($\ref{pde}$). To this end, we set part of the model's input to be $(t, \boldsymbol{x})$ and the model's output as $u(t, \boldsymbol{x})$. When we need to obtain the differential equation solution $u(t, \boldsymbol{x})$ at a new spacetime point in the solution domain that is different from the training point, we can directly call the model for output. Such a kind of design has another advantage. When solving a large class of differential equations, such as heat equations, we can quickly make full use of the information in the equation, as shown in Fig.~\ref{fig_diagram}E. 

On the other hand, we hope that the operator we learned is still valid when the equation parameters or initial/boundary conditions change, such as the diffusion coefficient or initial conditions in the diffusion equation. Towards this problem, we assume that for these given configurations, there exists a unique encoding $\boldsymbol{c}$ such that the solution directly depends on $(\boldsymbol{x}, t, \boldsymbol{c})$. As an illustration, the following theorem establishes the mathematical foundation for using neural networks to approximate this continuous mapping from the encoding $\boldsymbol{c}$ to the solution of the Cauchy problem of the diffusion equation.
\begin{theorem}\label{universal_approximation}
    Suppose $u_0 \in C(\Omega), f\in C(\Omega\times [0, T]),$ there exists a unique encoding $\boldsymbol{c} \in \mathbb{R}^m$ such that the solution of the Cauchy problem $(\ref{eu_diffusion})$ is continuously depending on $\boldsymbol{c}$, that is, $u(\boldsymbol{x}, t, \boldsymbol{c})$ is a continuous function with respect to $\boldsymbol{c}$. Then $\forall \; \epsilon >0$, there exists $N > 0$, $\{\alpha_j\}_{j=1}^N \in \mathbb{R}$, $\{\omega_j\}_{j=1}^N \subset \mathbb{R}^{n+1+m}$, and $\{b_j\}_{j=1}^N \subset \mathbb{R}$, such that the following defined function 
    \begin{equation}
        G(\boldsymbol{x}, t, \boldsymbol{c}) = \mathop{\sum}\limits_{j=1}^{N}\alpha_j \sigma(\omega_j^T \boldsymbol{z} + b_j), \; where \; \boldsymbol{z} = [\boldsymbol{x}^T, t, \boldsymbol{c}^T]^T \in \mathbb{R}^{n+1+m},
    \end{equation}
    satisfies
    \begin{equation}
        \mathop{sup}\limits_{(\boldsymbol{x}, t, \boldsymbol{c}) \in \mathbb{R}^{n+1+m}} |u(\boldsymbol{x}, t, \boldsymbol{c}) - G(\boldsymbol{x}, t, \boldsymbol{c})| < \epsilon.
    \end{equation}
    where $\sigma$ is a general non-polynomial function.
\end{theorem}
\begin{proof}
    Let $\boldsymbol{z} = [\boldsymbol{x}^T, t, \boldsymbol{c}^T]^T \in \mathbb{R}^{n+1+m}, K = \Omega\times \mathbb{R}\times\mathbb{R}^{m}$, then $u(\boldsymbol{x}, t, \boldsymbol{c})$ can be regarded as $u(\boldsymbol{z})$, which is continuous on $K$. Since $\sigma$ is a general non-polynomial function, it satisfies the requirements of Cybenko's theorem \cite{cybenko1989approximation}. Define a family of functions
    \begin{equation}
        \mathcal{G} = \Big\{G(\boldsymbol{z}) = \mathop{\sum}\limits_{j=1}^{N}\alpha_j \sigma(\omega_j^T\boldsymbol{z} + b_j): N\in \mathbb{N}, \alpha_j, \omega_j \in \mathbb{R}^{n+1+m}, b_j \in \mathbb{R}\Big\}.
    \end{equation}
    According to Cybenko's theorem or more general Hornik's theorem \cite{hornik1989multilayer}, $\mathcal{G}$ is dense in $C(K)$. That is $\forall \epsilon > 0,$ there exists $G \in \mathcal{G}$, satisfies
    \begin{equation}
        \mathop{sup}\limits_{(\boldsymbol{x}, t, \boldsymbol{c}) \in \mathbb{R}^{n+1+m}} |u(\boldsymbol{x}, t, \boldsymbol{c}) - G(\boldsymbol{x}, t, \boldsymbol{c})| < \epsilon.
    \end{equation}
\end{proof}

Following Theorem $\ref{universal_approximation}$, to ensure the assumptions are satisfied, we assume that we know the specific functional forms of the parameter $a(t,\boldsymbol{x})$ and the initial condition $u_0(\boldsymbol{x})$ in the model. Therefore, we can directly use the variable coefficients in $a(t,\boldsymbol{x})$ and $u_0(\boldsymbol{x})$ as the encoding $\boldsymbol{c}$, because they have unique encoded $a(t,\boldsymbol{x})$ and $u_0(\boldsymbol{x})$. Furthermore, there is a solid theoretical basis to tell us that the solution $u(t,\boldsymbol{x})$ continuously depends on these variables, thus ensuring that the solution $u_{\theta}$ we learned is continuously dependent on the encoding.

To enforce the continuity, we incorporate the differentiability constraints of the solution $u(t,\boldsymbol{x})$ on the variable of parameter $a(t,\boldsymbol{x})$ and the initial condition $u_0(\boldsymbol{x})$ into the loss function. This idea comes from the continuous dependence constraints of ODEs \cite{hartman2002ordinary} and PDEs \cite{evans2022partial}, and has achieved great success in plenty of concrete applications \cite{li2025improving}. We name the model as cd-PINN, whose loss function can be written as
\begin{equation}
    \mathcal{L}(\boldsymbol{\Theta}) = \lambda_{data}\mathcal{L}_{data} + \lambda_{res}\mathcal{L}_{res} + \lambda_{cd}\mathcal{L}_{cd},
\end{equation}
where
\begin{equation}
{\small
    \begin{split}
        & \mathcal{L}_{data} = \frac{1}{N_{data}} \mathop{\sum}\limits_{i=1}^{N_{data}}\Vert(\hat{\boldsymbol{u}}-\boldsymbol{u})(t_i,\boldsymbol{x}_i, \boldsymbol{c}_i)\Vert_2^2\\
        & \mathcal{L}_{res} = \frac{1}{N_r}\mathop{\sum}\limits_{i=1}^{N_r}\bigg\Vert\frac{d\hat{\boldsymbol{u}}}{dt}(t_i, \boldsymbol{x}_i,  \boldsymbol{c}_i) - \boldsymbol{\mathcal{P}}(t_i, \boldsymbol{x}_i, \hat{\boldsymbol{u}}_i, \boldsymbol{c}_i)\bigg\Vert_2^2 + \frac{1}{N_b}\mathop{\sum}\limits_{k=1}^{N_b}\Vert (\hat{\boldsymbol{u}}-\boldsymbol{g})(t_k, \boldsymbol{x}_k, \boldsymbol{c}_k)\Vert_2^2 \\
        & \qquad + \frac{1}{N_{0}}\mathop{\sum}\limits_{j=1}^{N_{0}}\Vert\hat{\boldsymbol{u}}(t_0, \boldsymbol{x}_j, \boldsymbol{c}_j) - \boldsymbol{u}_{0}(\boldsymbol{x}_j,\boldsymbol{c}_j)\Vert_2^2\\
        & \mathcal{L}_{cd} = \frac{1}{N_r}\mathop{\sum}\limits_{i=1}^{N_r}\bigg\Vert\Big(\frac{\partial^2\hat{\boldsymbol{u}}}{\partial \boldsymbol{c}\partial t}-\frac{\partial\boldsymbol{\mathcal{P}}}{\partial \hat{\boldsymbol{u}}}\frac{\partial\hat{\boldsymbol{u}}}{\partial \boldsymbol{c}}-\frac{\partial \boldsymbol{\mathcal{P}}}{\partial \boldsymbol{c}}\Big)(t_i, \boldsymbol{x}_i, \boldsymbol{c}_i)\bigg\Vert_2^2+\frac{1}{N_b}\sum_{k=1}^{N_b}\bigg\Vert\Big(\frac{\partial\hat{\boldsymbol{u}}}{\partial \boldsymbol{c}} - \frac{\partial \boldsymbol{g}}{\partial \boldsymbol{c}}\Big)(t_k, \boldsymbol{x}_k, \boldsymbol{c}_k)\bigg\Vert_2^2\\
        & + \frac{1}{N_0}\mathop{\sum}\limits_{j=1}^{N_0}\bigg\Vert\frac{\partial \hat{\boldsymbol{u}}}{\partial \boldsymbol{c}}(t_0, \boldsymbol{x}_j, \boldsymbol{c}_j) - \frac{\partial \boldsymbol{u}_0}{\partial \boldsymbol{c}}(\boldsymbol{x}_j, \boldsymbol{c}_j)\bigg\Vert_2^2,
    \end{split}}
    \label{loss-cd-PINN}
\end{equation}
where $\boldsymbol{c}$ denotes all the variable coefficients in $\boldsymbol{a}(t,\boldsymbol{x})$ and $\boldsymbol{u}_0(\boldsymbol{x})$. 

% The key difference between PICNO-S and PICNO is the inclusion of a new loss term $\mathcal{L}_{cd}$ that accounts for the continuous dependence (differentiable, to be exact).

It is worth noting that our model has a significant 
superiority over DeepONet and FNO. Once trained, it can be solved quickly in the full-time domain, allowing for arbitrarily high-resolution solutions to be obtained. Compared to the meta-learning-based PINN model, our model does not require fine-tuning and can be deployed immediately after training, similar to recent GPT models \cite{brown2020language, achiam2023gpt} for language modeling.

\section{Conclusion}

In this paper, based on the concept of continuous dependence of the PDE solution on its parameters and initial/boundary values, we propose cd-PINN for solving parameterized PDEs. Compared to existing operator learning frameworks such as DeepONet and FNO, our models require significantly less labeled data. Once trained, the proposed model enables rapid inference of the PDE solution at any point in time and space within the solution domain, thereby achieving high-resolution prediction. Unlike meta-learning-based PINN variants such as MAD-PINN, our proposed model can directly output the solution of a PDE after training without requiring retraining or fine-tuning on new configurations. This capability enables efficient deployment in application domains such as climate prediction and medical imaging, where complex PDEs with varying initial and boundary conditions or parameters need to be solved repeatedly and at different scales.

We numerically demonstrate that the proposed model outperforms benchmark methods such as FNO, DeepONet, and PI-DeepONet under small samples. Additionally, we validate the effectiveness of the explicit encoding and differentiability assumptions, as well as the performance of cd-PINN in high-dimensional problems. In the example of Burgers' equation, we compared the proposed cd-PINN with the Newton-Implicit FDM in terms of computational efficiency, and showed that cd-PINN is more efficient in large-scale simulations. Finally, we evaluate the model on problems where the theoretical assumptions are only approximately satisfied, and find that it remains effective when the solution exhibits transient stability over a short time interval.

There are several promising directions for future work, including extending the model to handle weak solutions of PDEs\cite{kharazmi2019variational}, which are common in problems involving shocks, discontinuous, or other non-smooth behaviors. Addressing such cases would significantly broaden the applicability of the method to real-world phenomena governed by conservation laws or nonlinear dynamics. Another valuable direction is to apply the model to practical domains such as weather forecasting\cite{pathak2022fourcastnet}, where data is often scarce, high-dimensional, and noisy.

\newpage

\backmatter

\bmhead{Code \& Data availability}

The source code and data for this project are available at \url{https://github.com/jay-mini/cd-PINN.git}.

\bmhead{Acknowledgements}

This work was supported by the National Key R\&D Program of China (Grant No. 2023YFC2308702), the National Natural Science Foundation of China (12301617), and the
Guangdong Provincial Key Laboratory of Mathematical and Neural Dynamical Systems (2024B1212010004).

\bmhead{Author Contributions}

Guojie Li: Investigation, Conceptualization, Methodology, Data curation, Formal analysis, Visualization, Writing-original draft. Wuyue Yang: Data Curation, Formal analysis, Visualization, Writing-original draft. Liu Hong: Supervision, Funding Acquisition, Conceptualization, Project Administration, Writing-Review \& Editing. All authors reviewed the manuscript.

\bmhead{Competing Interests}

Authors declare that they have no conflict of interest.

\newpage
\bibliography{sn-bibliography}


\end{document}